%% file: topological_derivative_final2.tex
\documentclass[12pt, twoside]{article}

\usepackage{tikz}
\usetikzlibrary{calc}

\usepackage{amsmath,amsthm}

\usepackage[utf8]{inputenc}

\usepackage{enumerate}
\usepackage{esint}

\usepackage{fancyhdr}
\usepackage{cite}
\usepackage{enumitem}

\usepackage{hyperref}
\hypersetup{%
    colorlinks=True, 
    citecolor=blue,
    linkcolor=black}

\usepackage[charter]{mathdesign}

\newcommand{\rom}[1]{\uppercase\expandafter{\romannumeral #1\relax}}

  \theoremstyle{definition}
  \newtheorem{theorem}{Theorem}[section]
  \newtheorem{corollary}[theorem]{Corollary}
  
  \newtheorem{lemma}[theorem]{Lemma}
  \newtheorem{definition}[theorem]{Definition}
  \newtheorem{remark}[theorem]{Remark}

  \newtheorem{assumption}{Assumption}

  \newtheorem*{assumption*}{Assumption}

  \newcommand{\fu}{u}
  \newcommand{\fp}{q}
  \newcommand{\fP}{Q}
  
  \numberwithin{equation}{section}
    
  \newcommand{\subjclass}[1]{\bigskip\noindent\emph{2010 Mathematics Subject Classification:}\enspace#1}
  \newcommand{\keywords}[1]{\noindent\emph{Keywords:}\enspace#1}

\input{slidesymbols}

\setenumerate[0]{label=(\alph*)}

\newcommand{\eps}{{\varepsilon}}

\newcommand{\overbar}[1]{\mkern 1.5mu\overline{\mkern-1.5mu#1\mkern-1.5mu}\mkern 1.5mu}

\newcommand{\ben}{\begin{equation}}
\newcommand{\een}{\end{equation}}

\newcommand{\benn}{\begin{equation*}}
\newcommand{\eenn}{\end{equation*}}

\usepackage[colorinlistoftodos]{todonotes}

\usepackage{authblk}

\begin{document}

\title{Topological sensitivities via a Lagrangian approach \\  for semilinear problems }

\author{Kevin Sturm \footnote{Technische Universit\"at Wien, Wiedner Hauptstr. 8-10,
		1040 Vienna, Austria, E-Mail: kevin.sturm@tuwien.ac.at}}

\date{}

\maketitle

\begin{abstract}
  In this paper we present a methodology that allows the efficient computation of the topological derivative for semilinear elliptic problems within 
  the averaged adjoint Lagrangian framework. The generality of our approach should also allow the extension to 
  evolutionary and other nonlinear problems. Our strategy relies on a rescaled differential quotient of the averaged adjoint state variable which we show converges weakly to a function satisfying an equation defined in the whole space. A unique feature and advantage of this framework is that we only need to work with weakly converging subsequences of the differential quotient. This allows the computation of the topological sensitivity within a simple functional analytic framework under mild assumptions.


\end{abstract}

  \maketitle

\section{Introduction}
\emph{Shape functions} (also called shape functionals) are real valued functions defined on sets of subsets of the Euclidean space $\VR^d$. The field of mathematics dealing with the minimisation of shape functions 
that are constrained by a partial differential equation is called \emph{PDE constrained shape optimisation}.  Numerous applications in the engineering and life sciences, such as the aircraft and car design or electrical impedance/magnetic induction tomography, underline its importance; \cite{a_HILANO_2011a,a_HILA_2008a}.  Among other approaches \cite{a_CLKU_2016a,a_BESI_2004a,b_NOSO_2013a,b_DEZO_2011a,b_SOZO_1992a} the topological derivative approach \cite{a_ESKOSC_1994a,a_SOZO_1999a,a_CEGAGUMA_2000a}  constitutes an important tool to solve shape optimisation problems for which the final topology of the shape is unknown. We refer to the monograph \cite{b_NOSO_2013a} and references therein for applications of this approach. 

The idea of the \emph{topological derivative} is to study the local behaviour of a shape function $J$ with respect to a family of singular perturbations  $(\Omega_\eps)$. Two important singular perturbations are obtained by translating and scaling of an inclusion $\omega$ which contains the origin by $\omega_\eps(z) := z+\eps \omega$; then the singular perturbations are given by $\Omega_\eps := \Omega \cup \omega_\eps(z)$ for $z\in \Omega^c$ and $\Omega_\eps := \Omega\setminus \overbar \omega_\eps(z)$ for $z\in \Omega$. Both singular perturbations are examples of the class of perturbations called dilatations that are considered in \cite{a_DE_2017a}. The topological derivative of a shape function $J$ with respect to perturbations $(\Omega_\eps)$ is defined by 
\ben\label{eq:topo_definition}
\partial J(\Omega) := \lim_{\eps \searrow 0} \frac{J(\Omega_\eps) - J(\Omega)}{\ell(\eps)},
\een
where $\ell:[0,\tau]\to \VR$, $\tau >0$, is an appropriate function depending on the perturbation chosen.  
If $\Omega$ is perturbed by a family of transformations $\Phi_\eps:=\Id + \eps \VV:\VR^d\to \VR^d$ generated by a Lipschitz vector field $\VV:\VR^d \to \VR^d$, that is, $\Omega_\eps := \Phi_\eps(\Omega)$, then we can choose $\ell(\eps)=\eps$ and \eqref{eq:topo_definition} reduces to the definition of the shape derivative \cite{b_SOZO_1992a}. So the topological derivative can be seen as an generalisation of the shape derivative.  In some cases, notably when shape functions are constrained by elliptic partial differential equations, the topological derivative can be obtained as the singular limit of the shape derivative as presented in the monograph \cite[pp. 12]{b_NOSO_2013a}. While the shape derivative can be interpreted as the Lie derivative on a manifold, the topological derivative is merely a semi-differential defined on a cone, which makes its computation a challenging topic; see \cite{a_DE_2017a}. 

The goal of this paper is to give a coincide way to compute topological sensitivities for the following class of semilinear problems. Given a bounded domain $\Dsf\subset \VR^d$, $d\in \{2,3\}$, with Lipschitz boundary $\partial \Dsf$ we want to find the topological derivative of the objective function 
\ben\left.\tag{C}
J(\Omega) := \int_{\Dsf} j(x,u(x))\; dx
\quad\right\}
\een
in an open set $\Omega\subset \Dsf$ subject to $\fu=\fu_\Omega$ solves the semilinear transmission problem
\ben\tag{S}
\left.
\begin{aligned}
  -\Div(\beta_1 \nabla \fu^+) + \varrho_1(\fu^+) &= f_1  \quad && \text{ in } \Omega \\
  -\Div(\beta_2 \nabla \fu^-) + \varrho_2(\fu^-) &= f_2 \quad && \text{ in } \Dsf\setminus \overbar{\Omega}\\
  \fu^-& = 0 \quad && \text{ on }\partial \Dsf \\
  (\beta_1\nabla \fu^+)\nu & = (\beta_2\nabla \fu^-)\nu \quad && \text{ on } \partial  \Omega \\
  \fu^+& = \fu^- \quad && \text{ on } \partial \Omega  
\end{aligned}\quad \right\}
\een
where $\fu^+,\fu^-$ denote the restriction of $\fu$ to  $\Omega$ and  $\Dsf\setminus \overbar \Omega$, respectively. The function $\nu$ denotes the outward pointing unit normal field 
along $\partial \Omega$. The technical assumptions for the matrix valued functions $\beta_1,\beta_2$ and the scalar functions $j, \varrho_1,\varrho_2,f_1,f_2$ will be introduced in Section~\ref{sec:4}. A related work is \cite[Ch. 10, pp. 277]{b_NOSO_2013a}, which is based on the research article \cite{a_IGNAROSOSZ_2009a}, where a semilinear problem without transmission conditions in a H\"older space setting is studied.

There are two main approaches to compute topological derivatives for PDE constrained shape functions. A typical and general strategy to obtain the topological sensitivity is to derive the asymptotic expansion of the partial differential equation with respect to the 
singular perturbation of the shape \cite{a_NASLSO_2005a,a_NASO_2003a}. For our problem above this would amount to prove that an expansion of the form (see \cite[p. 280]{b_NOSO_2013a})
\ben\label{eq:expansion1}
u_\eps(x) = u(x) + \eps K_1(\eps^{-1} x) + \eps^2(K_2(\eps^{-1}x) + u'(x)) + \Cr_\eps(x)
\een
exists. Here $K_1,K_2$ are so-called called boundary layer correctors, which solve certain exterior boundary value problems and $u'$ is called regular corrector and  solves a linearised system. The function $u_\eps$ denotes the solution to  (S) for the singular perturbed domain $\Omega_\eps$ and $\Cr_\eps(x)$ is an appropriate remainder. However, the proof of an expansion like \eqref{eq:expansion1} can technically involved and depends very much on the problem; \cite{a_IGNAROSOSZ_2009a}. 

A second approach to compute the topological derivative is presented in \cite{a_AM_2006a} and based on a perturbed adjoint equation, see also \cite{a_CEMOVO_1998a,a_AM_2006a,a_AMBO_2017a,c_GALA_2012a, phd_GA_2017a} and \cite{a_LHIFRSC_2013a}. A key of this method is to prove
\ben\label{eq:expansion2}
\begin{split}  
  u_\eps(x) & = u(x) + \eps K_1(\eps^{-1}x) + \Cr_\eps^1(x),\\
  p_\eps(x) & = p(x) + \eps \fP(\eps^{-1} x) + \Cr_\eps^2(x),
\end{split}
\een
where $K_1$ is the same as in \eqref{eq:expansion1}, $\fP$ is the solution to an exterior problem, and $\Cr_\eps^1,\Cr_\eps^2$ are appropriate remainder that have to go to zero in some function space. Here $p_\eps$ is the solution to a certain perturbed adjoint equation depending on the derivative of $J$; see \cite{a_AM_2006a}.  As a by-product of this 
approach one obtains the topological sensitivity for non-transmission type problems where 
Neumann boundary conditions on the 
inclusion are imposed. However, the proof of the expansions \eqref{eq:expansion2}, particularly for nonlinear problems, can be technically involved and necessitate knowledge of the asymptotic behaviour of $\fP$ and $K_1$ at infinity.  

In this paper we will show that neither the expansion \eqref{eq:expansion1} nor \eqref{eq:expansion2} are necessary to obtain the topological sensitivity for (S). For this purpose, we use a Lagrangian approach which uses the averaged adjoint variable $\fp_\eps$ \cite{a_ST_2015a, phd_ST_2014a,a_DEST_2017a}. The key ingredient, which leads to the existence of the topological derivative of (C), is the convergence property
\ben\label{eq:weak_Q_intro}
\nabla \left(\frac{\fp_\eps(z+\eps x) - \fp(z+\eps x)}{\eps}\right) \rightharpoonup \nabla \fP \qquad \text{ weakly in } L_2(\VR^d)^d, 
\een
where $\fP$ is the same function as in \eqref{eq:expansion2}. The averaged adjoint variable reduces to the usual adjoint in the unperturbed situation, that is, $q_0=q=p=p_0$. We emphasise that the weak convergence property \eqref{eq:weak_Q_intro} is a 
relaxation of \eqref{eq:expansion1} and \eqref{eq:expansion2}, since no remainder estimates are necessary. In addition no further knowledge about the asymptotic behaviour of $\fP$ at infinity is needed. We will demonstrate that the proof of \eqref{eq:weak_Q_intro} is constructive in that it reveals the equation $\fP$ must satisfy. This is particularly important when dealing with other more complicated nonlinear equations, e.g., quasilinear equations. We will show that our strategy also allows, with minor changes, to treat the extremal case where $\beta_1,\varrho_1,f_1=0$, i.e., the transmission problem (S) reduces to a semilinear equation with homogeneous Neumann boundary conditions on $\partial \Omega$. Compared to previous works we can prove the existence of the topological derivative under milder assumptions on the regularity of the inclusion.

\subsection*{Notation and definitions}
\paragraph{Notation for derivatives}
Let $ (\eps , \fu, \fp)\mapsto G(\eps,\fu,\fp): [0, \tau ] \times X \times Y \to \VR$ be a function defined on real normed vector spaces $X,Y$, and $\tau >0$. When the limits exist we use the following notation:
	\begin{gather}
	v\in X, \quad \partial_\fu G(\eps,u,\fp)(v):= \lim_{t \searrow 0}\frac{G(\eps,\fu+t v,\fp)-G(\eps,\fu,\fp)}{t}
	\\
	w\in Y, \quad \partial_\fp G(\eps,\fu,\fp)(w):= \lim_{t \to 0}\frac{G(\eps,\fu,\fp+t w)-G(\eps,\fu,\fp)}{t}.
	\end{gather}
The notation $t\searrow 0$ means that $t$ goes to $0$ by strictly positive values.

\paragraph{Miscellaneous notation}
Standard $L^p$ spaces and Sobolev spaces on an open set $\Omega\subset \VR^d$ are denoted $L_p(\Omega)$ and $W^k_p(\Omega)$, respectively, where $p\ge 1$ and $k\ge 1$. In case $p=2$ and $k\ge 1$  we set as usual $H^k(\Omega):= W^k_2(\Omega)$. Vector valued spaces are denoted $L_p(\Omega)^d:=L_p(\Omega,\VR^d)$ and $W^k_p(\Omega)^d:=W^k_p(\Omega,\VR^d)$.  We write $\fint_A f\; dx := \frac{1}{|A|} \int_{A} f\; dx$ to indicate the average of $f$ over a measurable set $A$ with measure $|A|<\infty$. The \emph{H\"older conjugate} of $p\in [1,\infty)$ is defined by $p':= p/(p-1)$.  For $1\le p<d$ we denote by $p^*:= dp/(d-p)$ the \emph{Sobolev conjugate} of $p$. Given a normed vector space $V$ we denote by $\mathcal L(V,\VR)$ the space of linear and continuous functions on $V$.  We denote by $B_\delta(x)$ the ball centred at $x$ with radius $\delta >0$ and  set $\bar B_\delta(x) :=\overbar{B_\delta(x)}$.

\section{Abstract averaged adjoint framework}\label{sec:2}

\subsection{Lagrangians and infimum}
The following material can be found in \cite{a_DEST_2017a}. We begin with the definition of a Lagrangian function.
\begin{definition}
 Let $X$ and $Y$ be vector spaces and $\tau >0$. A \emph{parametrised Lagrangian} (or short Lagrangian) is a function
\begin{gather*}
(\eps,\fu,\fp) \mapsto  G(\eps,\fu,\fp): [0, \tau ] \times X \times Y \to \VR,
\end{gather*}
satisfying for all $(\eps, \fu)\in [0,\tau]\times X$, 
\ben
\fp\mapsto G(\eps,\fu,\fp) \quad \text{ is }  \emph{affine} \text{ on } Y.
\een
\end{definition}
The next definition formalises the notion of state and perturbed state variable associated with $G$. 
\begin{definition}
  For $\eps \in [0,\tau]$ we define the \emph{state equation} by: find $u_\eps \in X$, such that 
  \ben\label{eq:state}
  \text{ find }  u_\eps\in X \text{ such that } \quad  \partial_\fp G(\eps,\fu_\eps,0)(\varphi)=0 \quad \text{ for all } \varphi\in X. 
\een
The set of solution of \eqref{eq:state} (for $\eps$ fixed) is denoted by $E(\eps)$. For $\eps =0$, the elements of $E(\eps)$ are called \emph{unperturbed states} (or short states) and for $\eps >0$ they are referred to as \emph{perturbed states}. 
\end{definition}

\begin{definition}
We introduce for $\eps \in [0,\eps]$ the set of minimisers
\ben
X(\eps) = \{u_\eps \in E(\eps): \; \inf_{u\in E(\eps)} G(\eps,u, 0) = G(\eps,u_\eps,0) \}.
\een
\end{definition}
Notice that $X(\eps)\subset E(\eps)$ and that $X(\eps)=E(\eps)$ whenever $E(\eps)$ is a singleton.
We associate with the \emph{parameter} $\eps$ the \emph{parametrised infimum}
\ben\label{def_g}
  \eps \mapsto g(\eps) := \inf_{\fu \in E(\eps)} G(\eps ,\fu,0):[0,\tau]\to \VR.
\een 
 We now recall sufficient conditions introduced in \cite{a_DEST_2017a} under which the limit    \ben
    d_{\ell}g(0):= \lim_{\eps\searrow 0}\frac{g(\eps)-g(0)}{\ell(\eps)}
    \een
    exists, where $\ell: [0,\tau] \to \VR$ is a given function satisfying $\ell(0)=0$ and $\ell(\eps)>0$ for 
$\eps \in (0,\tau]$. 
The key ingredient is the so-called \emph{averaged adjoint equation}. The definition of the 
averaged adjoint equation requires that the set of states is nonempty: 
\begin{assumption*}[H0]
  For all $\eps\in [0,\tau]$ the set $X(\eps)$ is nonempty. 
\end{assumption*}
 Before we can introduce the averaged adjoint equation we need the following hypothesis. 
\begin{assumption*}[H1]
  For all $\eps \in [0,\tau]$ and $(\fu_0, \fu_\eps) \in X(0)\times X(\eps)$ we assume:
\begin{enumerate}
      	\item [\textup{(i)}] 
	  For all $\fp \in Y$, the mapping $s\mapsto G(\eps , s\fu_\eps + (1-s)\fu_0), \fp):[0,1]\to \VR$ is absolutely continuous.
	\item [\textup{(ii)}]
	  For all $(\varphi,q)\in X\times Y$ and almost all $s\in (0,1)$
	the function 
	\ben
	s \mapsto \partial_\fu G(\eps, s\fu_\eps + (1-s)\fu_0,\fp)(\varphi):[0,1]\to \VR
	\een
	 is well-defined and belongs to $L_1(0,1)$.
\end{enumerate}
\end{assumption*}
\begin{remark}
Notice that item (i) implies that for all $\eps \in [0,\tau]$, $(\fu_0, \fu_\eps)\in X(0)\times X(\eps)$ and $\fp\in Y$, 
\ben
\label{eq.int}
	G(\eps,\fu_\eps,\fp) = G(\eps,\fu_0,\fp) +\int_0^1  \partial_\fu G(\eps, s\fu_\eps + (1-s)\fu_0,\fp)(\fu_\eps-\fu_0) \, ds.
\een
This follows at once by applying the fundamental theorem of calculus to $s\mapsto G(\eps, s\fu_\eps + (1-s)\fu_0, \fp)$ on $[0,1]$.
\end{remark}

The following gives the definition of the averaged adjoint equation; see \cite{c_ST_2015a}. 
\begin{definition}
  Given $\eps\in [0,\tau]$ and $(\fu_0,\fu_\eps)\in X(0)\times X(\eps)$, the 
	\emph{averaged adjoint state equation} is defined as follows: find $\fp_\eps \in X$, such that
	\ben\label{eq:aa_equation}
	\int_0^1 \!\!\partial_\fu G(\eps,s\fu_\eps + (1-s)\fu_0 ,\fp_\eps)(\varphi)\, ds=0 \quad \text{ for all } \varphi\in X. 
	\een
	For every triplet $(\eps, u_0,u_\eps)$ the set of solutions of \eqref{eq:aa_equation} is denoted by $Y(\eps,\fu_0,\fu_\eps)$ and its elements are referred to as \emph{adjoint states} for $\eps =0$ and \emph{averaged adjoint states} for $\eps >0$. 
\end{definition}

	Notice that $ Y(0, \fu_0) := Y(0,\fu_0,\fu_0)$ is the usual set of \emph{adjoint states} associated with $\fu_0$,
	\begin{gather}
	\label{eq.averageAdjoint0}
	Y(0,\fu_0) = \left \{\fp\in Y: 
	\forall \varphi \in X, \,
	\partial_\fu G(0, \fu_0,\fp)(\varphi)=0
	\right \}.
	\end{gather}
An important consequence of the introduction of the averaged adjoint state is
the following identity: for all $\eps\in [0,\tau]$, $(\fu_0, \fu_\eps)\in  X(0) \times X(\eps)$ and  $\fp_\eps \in Y(\eps,\fu_0,\fu_\eps)$,
\begin{gather}
\label{eq.simple}
	g(\eps)=G(\eps,\fu_\eps,\fp_\eps)  = G(\eps,\fu_0,\fp_\eps).
\end{gather}
This is readily seen by substituting $\fp_\eps$ into  equation \eqref{eq.int}. 
The following result is an extension of \cite[Thm. 3.1]{a_DEST_2017a}. We refer the reader to \cite{c_DEST_2016a,c_ST_2015a} for further
results on the averaged adjoint approach and \cite{phd_ST_2014a} for more examples involving the shape derivative. 
\begin{theorem}[\cite{a_DEST_2017a}] \label{thm:diff_lagrange}
  Let Hypothoses~(H0) and (H1) and the following conditions be satisfied. 
  \begin{itemize}
    \setlength{\itemsep}{3pt}
  \item[(H2)]  For all $\eps\in [0,\tau]$ and $(\fu_0,\fu_\eps) \in X(0)\times X(\eps)$ the set $Y(\eps,\fu_0,\fu_\eps)$ is nonempty.
  \item[(H3)] For all $\fu_0\in X(0)$ and $\fp_0\in Y(0,\fu_0)$ the limit 
\ben
\partial_\ell G(0,\fu_0,\fp_0) := \lim_{\eps\searrow 0}\frac{G(\eps,\fu_0, \fp_0) - G(0,\fu_0, \fp_0)}{\ell(\eps)} \quad \text{ exists}.
\een
\item[(H4)]
	There exist sequences $(u_\eps)$ and $(\fp_\eps)$, where $u_\eps \in X(\eps)$ and $\fp_\eps \in Y(\eps,u_0,u_\eps)$, such that the limit 
        \ben
	R := \lim_{\eps\searrow 0} \frac{G(\eps,\fu_0,\fp_\eps)-G(\eps,\fu_0,\fp_0)}{\ell(\eps)} \quad \text{ exists}. 
	\een
\end{itemize}    
Then we have 
\ben
d_\ell g(0) = \partial_\ell G(0,\fu_0,\fp_0) + R.
\een
Moreover, $R=R(\fu_0,\fp_0)$ does not depend on the choice of the sequences $(\fu_\eps)$ and $(\fp_\eps)$, but only on $\fu_0$ and $\fp_0$. 
\end{theorem}
\begin{proof}
  Thanks to Hypothoses~(H0)-(H2) the sets $X(\eps)$ and $Y(\eps,\fu_0,\fu_\eps)$ are nonempty for all $\eps$. 
Therefore in view of \eqref{eq.simple} we have for all $\eps \in [0,\tau]$,  $(u_0,u_\eps) \in X(0)\times X(\eps)$ and $\fp_\eps \in Y(\eps,u_\eps,u_0)$,   
\ben
\begin{split}
g(\eps)-g(0) & =  G(\eps,\fu_0,\fp_\eps)-G(0,\fu_0,\fp_0)\\
	  & =  G(\eps,\fu_0,\fp_\eps) - G(\eps,\fu_0,\fp_0) + G(\eps,\fu_0,\fp_0)-G(0,\fu_0,\fp_0).
\end{split}
\een
Thus selecting $(u_\eps)$ and $(\fp_\eps)$ from Hypothosis (H4) and using Hypothosis (H3) we obtain  
\ben\label{eq:dg_}
\begin{split}
  d_\ell g(0) & = \lim_{\eps\searrow 0} \frac{G(\eps,\fu_0,\fp_0)-G(0,\fu_0,\fp_0)}{\ell(\eps)} + \lim_{\eps\searrow 0}  \frac{G(\eps,\fu_0,\fp_\eps)-G(\eps,\fu_0,\fp_0)}{\ell(\eps)} \\
& =\partial_\ell G(0,\fu_0,\fp_0) + R.
\end{split}
\een
It follows from \eqref{eq:dg_} that $R$ only depends on $\fu_0$ and $\fp_0$. 
\end{proof}

\begin{remark}
  An important application of Theorem~\ref{thm:diff_lagrange} is the computation of shape derivatives for which one chooses
  $\ell(\eps) =\eps$, see e.g., \cite{c_ST_2015a,phd_ST_2014a}. In this case one typically has $R(\fu_0,\fp_0)=0$, which means 
      \ben
d_\eps g(0) = \partial_\eps G(0,\fu_0,\fp_0).
       \een
       However for the topological derivative, in which case $\ell(\eps)\neq \eps$, the term $R(\fu_0,\fp_0)$ is typically not equal to zero as shown by the one dimensional example of \cite{c_DEST_2016a}.
\end{remark}

\section{Linear elliptic equations in \texorpdfstring{$\VR^d$}{Lg}}\label{sec:3}
In preparation for the study of the semilinear problem (S), we first recall existence and uniqueness results for the following exterior problem.  Let $\omega\subset \VR^d$ be an open and bounded set, and let $\zeta\in \VR^d$ be a vector. 
Given a suitable vector space $V$ of functions $\VR^d\to \VR$ we consider: find $Q_\zeta\in V$ such that
\ben\label{eq:laplace_unbounded_intro}
\int_{\VR^d} A  \nabla \psi \cdot \nabla Q_{\zeta} \;dx  =   \int_{\omega} \zeta \cdot \nabla \psi \; dx \quad \text{ for all } \psi \in V.
\een
Here $A:\VR^d \to \VR^{d\times d}$ is a measurable, uniformly coercive (not necessarily symmetric) matrix-valued 
functions, that is, there are constants $M_1,M_2>0$, such that  
\ben
M_1|v|^2 \le A(x)v\cdot v \le M_2|v|^2 \quad \text{  for a.e } x\in \VR^d \text{ and all } v\in \VR^d. 
\een
The well-posedness of \eqref{eq:laplace_unbounded_intro} can be achieved by several choices of $V$. The most popular ones are weighted Sobolev spaces; see \cite{a_DELI_1955a}.  In the next section we discuss a more straight forward choice for $V$.

\subsection{Solution in the Beppo-Levi space}

\begin{definition}
For $d\ge 1$ define 
  \ben
 BL(\VR^d) :=  \{u\in H^1_{\text{loc}}(\VR^d):\; \nabla u \in L_2(\VR^d)^d\}.
  \een
  Then the \emph{Beppo-Levi space} is defined by 
     \ben
      \dot{BL}(\VR^d) := BL(\VR^d)/\VR,
      \een
      where $/\VR$ means that we quotient out the constant functions. We denote by $[u]$ the equivalence classes of $\dot{BL}(\VR^d)$. The Beppo-Levi space is equipped with the norm
      \ben
      \|[u]\|_{\dot H^1(\VR^d)} := \|\nabla u\|_{L_2(\VR^d)^d}, \quad u\in [u].
      \een
\end{definition}
The Beppo-Levi space is a Hilbert space (see \cite{a_DELI_1955a,a_ORSU_2012a} and also \cite{a_AUMEPR_2015a}) and $C^\infty_c(\VR^d)/\VR$ is dense in $\dot{BL}(\VR^d)$.

\begin{lemma}
  Let $d\ge 1$ and suppose that $A$ satisfies \eqref{eq:laplace_unbounded_intro}. Then there exists a unique equivalence class $[Q]\in \dot{BL}(\VR^d)$ solving 
  \ben\label{eq:exterior_beta0}
 \int_{\VR^d} A \nabla \psi\cdot \nabla \fP  \; dx = \int_\omega \zeta \cdot \nabla \psi\; dx \quad \text{ for all } \psi \in BL(\VR^d).
 \een
\end{lemma}
\begin{proof}
This is a direct consequence of the theorem of Lax-Milgram.
\end{proof}
As shown in \cite{a_ORSU_2012a}, in dimension $d\ge 3$ every equivalence class $[u]$ of $\dot{BL}(\VR^d)$ contains an element $u_0\in [u]$ that is in turn contained in the Banach space 
\ben
\VE_2(\VR^d) := \{ u\in L_{2^*}(\VR^d):\; \nabla u\in L_2(\VR^d)^d\} 
\een
equipped with the norm 
\ben
\|u\|_{\VE_2} := \|u\|_{L_{2^*}} + \|\nabla u\|_{(L_2)^d}. 
\een
This follows at once since $C^\infty_c(\VR^d)/\VR$ is dense in $\dot{BL}(\VR^d)$
 and from the Gagliardo-Nirenberg-Sobolev inequality; see \cite{a_ORSU_2012a}. As a result for 
 $d \ge 3$ we can replace the Beppo-Levi space by $\VE_2(\VR^d)$ and can even consider a more general problem. 
 \begin{lemma}\label{lem:exitence_E}
   Let $d\ge 3$. Suppose that $A$ satisfies \eqref{eq:laplace_unbounded_intro} and $A=A^\top$ a.e. on $\VR^d$. Then for every $F\in \mathcal L(\VE_2(\VR^d),\VR)$ there exists a unique solution $Q \in \VE_2(\VR^d)$ to 
   \ben\label{eq:sol_E2}
\int_{\VR^d} A  \nabla \psi \cdot \nabla Q \;dx  =  F(\psi) \quad \text{ for all } \psi \in \VE_2(\VR^d). 
\een
 \end{lemma}
 \begin{proof}
A proof can be found in the appendix. 
 \end{proof}

 So for $d\ge 3$ equation \eqref{eq:laplace_unbounded_intro} admits a unique solution in $V:= \VE_2(\VR^d)$ since obviously $F(\varphi) := \int_\omega \zeta \cdot \nabla \varphi \;dx \in \mathcal L(\VE_2(\VR^d),\VR)$.

\subsection{Relation to weighted Sobolev spaces}
Since the exterior equation \eqref{eq:laplace_unbounded_intro} is, as we will see later, of paramount importance for the first topological derivative we review here an alternative choice for the space $V$, namely, a weighted Sobolev Hilbert space. We follow the presentation of \cite{a_AMNOANVA_2014a}, where a more general situation than the following is considered. 

For this purpose we introduce the weight function $w\in L_1(\VR^d)$ defined by 
\ben
w(x) := (1+|x|^2)^{-\gamma_d}
\een
where $\gamma_d := \frac{d}{2} + \delta$ and $\delta \in (0,1/2)$ is arbitrary, but fixed. Since the weight satisfies $|w|^p \le |w|$ on $\VR^d$ for $p\in [1,\infty)$ it also follows that $w\in L_p(\VR^d)$ for all $p\in [1,\infty)$.

\begin{definition}
  The weighted Hilbert Sobolev space $H^1_w(\VR^d)$ 
is defined by  
\ben
H^1_w(\VR^d) := \{u:\VR^d \to \VR \text{ measurable}: \; \sqrt wu \in L_2(\VR^d), \quad \nabla u \in L_2(\VR^d)^d\}. 
\een
The norm on $H^1_w(\VR^d)$ is given by $\|u\|_{H^1_w} := \|\sqrt wu\|_{L_2} + \|\nabla u\|_{(L_2)^2}$.
\end{definition}

The weight $w$ is chosen in such a way that the set of constant functions on $\VR^d$ are contained 
in $H^1_w(\VR^d)$.  Therefore it is clear that \eqref{eq:laplace_unbounded_intro} can only be uniquely solvable in $H^1_w(\VR^d)$ up to a constant. A remedy is to consider the quotient space 
\ben
\dot H^1_w(\VR^d) :=  H^1_w(\VR^d)/{\VR}
\een
and equip this space, as in \cite{a_AMNOANVA_2014a}, with the quotient norm
  \ben
  \|[u]\|_{\dot H^1_w} := \inf_{c\in \VR} \|u+c\|_{H^1_w(\VR^d)},
  \een
  where $[u]$ denote the equivalence classes of $\dot H^1_w(\VR^d)$. In \cite[Cor. C.5, p. 23]{a_AMNOANVA_2014a} it is shown that there is a constant $c>0$, such that, 
  $\|[u]\|_{\dot H^1_w} \le c\|\nabla u\|_{(L_2)^d}$ for all $u\in H^1_w(\VR^d)$.  Therefore existence of a solution to \eqref{eq:laplace_unbounded_intro} follows directly from the theorem of Lax-Milgram. 

In the following lemma let us a agree that the Sobolev conjugate of $2$ in dimension two 
is given by $\infty$, i.e. $2^* := \infty$ if $d=2$. 
\begin{lemma}
  We have $\VE_2(\VR^d) \hookrightarrow H^1_w(\VR^d)$ for all $d \ge 2$, i.e., there is a constant $C>0$, such that
  \ben
  \|u\|_{H^1_w} \le C \|u\|_{\VE_2} \quad \text{ for all } u\in \VE_2(\VR^d).  
  \een
\end{lemma}
\begin{proof}
  Let $u\in \VE_2(\VR^d)$ be given so that $u\in L_{2^*}(\VR^d)$. In case $d\ne 2$, we have $2^* = \frac{2d}{d-2}$. Therefore the H\"older conjugate of 
  $2^*/2$ is given by $\frac{2^*}{2^*-2}=d/2$ and H\"older's inequality yields
  \ben\label{eq:est_p}
  \int_{\VR^d} w u^2\; dx \le \|w\|_{L_{d/2}(\VR^d)} \|u\|_{L_{p^*}(\VR^d)}^2.
  \een
  Since $d \ge 2$ we deduce $\sqrt w u \in L_2(\VR^d)$ and since by definition also $\nabla u \in L_2(\VR^d)^d$ we deduce $\VE_2(\VR^d)\subset H^1_w(\VR^d)$ and the continuity of the embedding follows from \eqref{eq:est_p}.
  In case $d=2$ we have $2^* = \infty$ and thus  H\"older's inequality directly gives \eqref{eq:est_p} and thus the continuous embedding. 
\end{proof}

\section{The topological derivative  via Lagrangian}\label{sec:4}
In this section we show how Theorem~\ref{thm:diff_lagrange} of Section~\ref{sec:2} can be used to compute the topological derivative for a semilinear transmission problem. Our approach is related to the one of \cite{a_AM_2006a} (see also \cite{phd_AM_2003a}), where also a perturbed adjoint equation is used, too. However the main difference here is that we only need to work with weakly converging subsequences and do not need to know any asymptotic behaviour of the limiting function. 

\subsection{Weak formulation and apriori estimates}
In the following exposition we restrict ourselves to the shape function
\ben\label{eq:cost_simple}
J(\Omega) = \int_{\Dsf} u^2 \;dx,
\een
where $u=u_\Omega\in H^1_0(\Dsf)\cap L_\infty(\Dsf)$ is the weak solution of (S): 
\ben\label{eq:state_weak_chi}
\int_{\Dsf} \beta_\Omega \nabla u\cdot \nabla \varphi + \varrho_\Omega(u)\varphi \; dx = \int_{\Dsf} f_\Omega\varphi\;dx \quad \text{ for all }  \varphi \in H^1_0(\Dsf),
\een
where $\beta_\Omega:\VR^d \to \VR^{d\times d}$ and $f_\Omega:\VR^d\to \VR$ are defined by 
\ben
\beta_\Omega(x) := \left\{
  \begin{array}{cl}
    \beta_1(x) & \text{ for } x\in \Omega\\
    \beta_2(x) & \text{ for } x\in \VR^d\setminus \overbar \Omega
\end{array}\right.,  \quad f_\Omega(x) := \left\{
  \begin{array}{cl}
    f_1(x) & \text{ for } x\in \Omega\\
    f_2(x) & \text{ for } x\in \VR^d\setminus \overbar \Omega
\end{array}\right., 
\een
and similarly $\varrho_\Omega$ is defined by
\ben
\varrho_\Omega(u) := \left\{
  \begin{array}{cl}
   \varrho_1(u) & \text{ for } x\in \Omega\\
   \varrho_2(u) & \text{ for } x\in \VR^d\setminus \overbar \Omega.
\end{array}\right.
\een
Notice that $\beta_\Omega = \beta_1 \chi_{\Omega} + \beta_2 \chi_{\VR^d\setminus \Omega}$, $f_\Omega = f_1 \chi_{\Omega} + f_2 \chi_{\VR^d\setminus \Omega}$ and 
$\varrho_\Omega(u) = \varrho_1(u)\chi_{\Omega} + \varrho_2(u) \chi_{\VR^d\setminus \Omega}$. 

It can be checked that the following proofs remain true when the shape function \eqref{eq:cost_simple} is replaced by (C) from the introduction under the assumption that $j$ 
is sufficiently smooth. However, in favour of a clearer presentation we use 
the simplified cost function \eqref{eq:cost_simple}. The functions $\beta_i, \varrho_i$ and $f_i$ are specified in the following assumption. The extremal case where $\beta_1,\varrho_1,f_1$ are zero will be discussed in the last section. 
\begin{assumption}\label{ass:varrho_monotone}
  \begin{itemize}\setlength{\itemsep}{1pt}
    \item[(a)] For $i=1,2$, we assume that $\beta_i\in C^1(\VR^d)^{d\times d}$ and that there are constants $\beta_m,\beta_M>0$, such that 
      \ben\label{eq:bound_beta}
      \beta_m|v|^2 \le \beta_i(x)v\cdot v \le \beta_M|v|^2 \quad \text{ for all } x\in \VR^d,\; v\in \VR^d.
      \een
    \item[(b)] For $i=1,2$, we assume $\varrho_i \in C^1(\VR)$, $\varrho_i(0)=0$ and the monotonicity condition
\ben\label{eq:monotone_varrho}
(\varrho_i(x)-\varrho_i(y))(x-y) \ge 0 \quad \text{ for all } x,y\in \VR. 
\een
\item[(c)] For $i=1,2$, we assume $f_i\in H^1(\Dsf)$ if $f_1=f_2$ and $f_i\in H^1(\Dsf)\cap C(\Dsf)$ if $f_1\ne f_2$. 
\end{itemize}
\end{assumption}

Notice that since for $x\in \Dsf$ the matrix $\beta_\Omega(x)$ is either equal to $\beta_1(x)$ or $\beta_2(x)$ and in view of the bound \eqref{eq:bound_beta}, we have 
\ben\label{eq:betam}
\beta_m|v|^2 \le \beta_\Omega(x)v\cdot v \quad \text{ for all } x\in \VR^d, \; v\in \VR^d.
\een
Similarly, in view of the monotonicity property \eqref{eq:monotone_varrho} and $\varrho_i(0)=0$, we get  
\ben\label{eq:varphim}
0 \le \varrho_\Omega(x)x \quad \text{ for all } x\in \VR^d.
\een
\begin{lemma}\label{lem:aprior_u}
  \begin{itemize}
    \item[(i)] Let $f\in L_r(\Dsf)$, $r>d/2$. Then for every measurable set $\Omega\subset \Dsf$ there is a unique solution $u_\Omega$ of \eqref{eq:state_weak_chi}. Moreover, there is a constant $C$ independent of $\Omega$, such that 
      \ben\label{eq:bound_u}
      \|u_\Omega\|_{L_\infty(\Dsf)} +  \|u_\Omega\|_{H^1_0(\Dsf)}\le C\|f\|_{L_r(\Dsf)}.
      \een
    \item[(ii)]
      For every $z\in \Dsf\setminus \overbar \Omega$, we find $\delta>0$, such that $u_{\Omega}\in H^3(B_\delta(z))$.
  \end{itemize}
\end{lemma}
\begin{proof}
  (i) Our assumptions imply that we can apply \cite[Theorem~4.5]{b_TR_2005a} which gives the existence of a solution to \eqref{eq:state_weak_chi} and also the apriori bound \eqref{eq:bound_u}. As pointed out in this reference the constant $C$ is independent of the nonlinearity $\varrho_\Omega$. 

(ii) 
Let $U := \Dsf\setminus \overbar \Omega$ and $z\in U$. The restriction of $u$ to $U$ solves 
\ben\label{eq:state_weak_chi1}
\int_{U} \beta_2 \nabla u\cdot \nabla \varphi \; dx = \int_{U} \tilde f\varphi\;dx\quad \text{ for all }  \varphi \in H^1_0(U),
\een
with right-hand side $\tilde f(x) := f_2(x) - \varrho_2(u(x))$. Since $\nabla \tilde f = \nabla f_2 - \varrho_2'(u)\nabla u \in L_2(U)^d$ we have $\tilde f\in  H^1(U)$. Hence $u\in H^3_{\text{loc}}(U)$ by standard regularity theory for elliptic PDEs; see, e.g.,  \cite[Thm. 2, p. 314]{b_EV_1998a}. Since $U$ is open we can choose $\delta>0$ such that $B_\delta(z)\Subset U$. This finishes the proof.   
\end{proof}

\begin{remark}
  Although we restrict ourselves to Dirichlet boundary conditions in (S) other boundary conditions, e.g., Neumann boundary conditions, can be considered as well. This only requires minimal changes in the following analysis and we will make remarks at the relevant places.  
\end{remark}

\subsection{The parametrised Lagrangian}
From now on we fix:
\begin{itemize}\setlength\itemsep{0.3em}
  \item an open and bounded set $\omega\subset \VR^d$ with $0\in \omega$,
  \item an open set $\Omega \Subset \Dsf$ and a point $z\in \Dsf\setminus \overbar \Omega$,
  \item the perturbation $\Omega_\eps := \Omega\cup \omega_\eps(z)$, where $\omega_\eps(z) := z+\eps \omega$ and $\eps\in[0,\tau]$, $\tau >0$.
\end{itemize}
To simplify notation we will often write $\omega_\eps$ instead of $\omega_\eps(z)$. Let $X=Y=H^1_0(\Dsf)$ and introduce the Lagrangian $ G:[0,\tau]\times X\times Y \to \VR$ associated with the perturbation $\Omega_\eps$ by 
\ben\label{eq:lagrange_scale_eps}
G(\eps,\fu,\fp) := \int_{\Dsf} \fu^2\;dx + \int_{\Dsf} \beta_{\eps} \nabla \fu \cdot \nabla\fp +  \varrho_\eps(\fu)\fp \;dx - \int_{\Dsf} f_\eps\fp \;dx,
\een
where we use the abbreviations
\ben\label{eq:abb_beta}
\begin{split}
  \beta_\eps  := \beta_1 \chi_{\Omega_\eps} + \beta_2 \chi_{\VR^d\setminus  \Omega_\eps},& \quad  \quad f_\eps  := f_1 \chi_{\Omega_\eps} + f_2 \chi_{\VR^d\setminus  \Omega_\eps}, \quad \varrho_\eps(u) := \varrho_1(u) \chi_{\Omega_\eps} + \varrho_2(u) \chi_{\VR^d\setminus \Omega_\eps}.
\end{split}
\een
We are now going to verify that Hypotheses~(H0)-(H4) are satisfied with $\ell(\eps) = |\omega_\eps|$. Moreover, we will determine the explicit form of $R(u,p)$. 

\begin{remark}[Removing an inclusion]
  We only treat the case of "adding" a hole here, i.e., $\Omega_\eps := \Omega\cup \omega_\eps(z)$ for $z\in \Dsf\setminus \overbar \Omega$. The second case of "removing" a hole, i.e., $\Omega_\eps := \Omega \setminus \overbar \omega_\eps(z)$ for $z\in \Omega$ can be dealt with in the same way.  
\end{remark}

\subsection{Analysis of the perturbed state equation}
The \emph{perturbed state equation} reads: find $u_\eps \in H^1_0(\Dsf)$ such that
\ben
\partial_\fp G(\eps, u_\eps, 0)(\varphi)=0 \quad \text{ for all } \varphi \in H^1_0(\Dsf),
\een
or equivalently $u_\eps \in H^1_0(\Dsf)$ satisfies,
\ben\label{eq:PDE_constraint_topo}
\int_{\Dsf} \beta_{\eps} \nabla u_\eps\cdot \nabla\varphi + \varrho_\eps(u_\eps)\varphi \;dx = \int_{\Dsf} f_\eps\varphi \;dx \quad \text{ for all } \varphi \in H^1_0(\Dsf).
\een
Henceforth we write $u:=u_0$ to simplify notation. Since \eqref{eq:PDE_constraint_topo} is precisely \eqref{eq:state_weak_chi}
with $\Omega = \Omega_\eps$, we infer from Lemma~\ref{lem:aprior_u} that \eqref{eq:PDE_constraint_topo} admits a unique solution. This means that $E(\eps) = \{u_\eps\}$ is a singleton and thus $E(\eps)=X(\eps)$ and Hypothesis~(H0) is satisfied. From this and Assumption~\ref{ass:varrho_monotone} we also infer that Hypothesis~(H1) is 
satisfied. We proceed by shoing a H\"older-type estimate for $(u_\eps)$.  
\begin{lemma}\label{lem:u_ueps}
	There is a constant $C>0$, such that for all small $\eps>0$, 
	\ben\label{eq:est_u_eps_D}
	\|u_\eps - u\|_{H^1(\Dsf)} \le C\eps^{d/2}. 
	\een
\end{lemma}
\begin{proof}
   We obtain from \eqref{eq:PDE_constraint_topo}
	\ben\label{eq:diff_u_ueps}
	\begin{split}
	  \int_{\Dsf} \beta_{\eps} \nabla (u_\eps-u)\cdot \nabla\varphi \;dx + \int_{\Dsf} (\varrho_\eps(u_\eps) - \varrho_\eps(u)) \varphi\; dx   = &  - \underbrace{\int_{\omega_\eps} (\beta_1-\beta_2)\nabla u\cdot \nabla \varphi\;dx}_{=:\rom{1}(\eps,\varphi)} \\
																		     & - \underbrace{\int_{\omega_\eps} (\varrho_1(u)-\varrho_2(u)) \varphi\; dx}_{=:\rom{2}(\eps,\varphi)}\\
																		     & + \underbrace{\int_{\omega_\eps} (f_1-f_2) \varphi\; dx}_{=:\rom{3}(\eps,\varphi)}
        \end{split}
	\een
	for all $\varphi \in H^1_0(\Dsf)$. Hence, since $u\in C^1(\overbar B_\delta(z))$ for $\delta >0$ sufficiently small, we can apply H\"older's inequality to obtain
	\ben\label{eq:I_III}
         \begin{split}
	   |\rom{1}(\eps, \varphi)| & \le \|\beta_1-\beta_2\|_{C(\overbar B_\delta(z))^{d\times d}} \|\nabla u\|_{C(\overbar B_\delta(z))^d} \sqrt{|\omega_\eps|}\|\nabla \varphi\|_{L_{2}(\Dsf)^d}\\
	   |\rom{2}(\eps, \varphi)| & \le  \|\varrho_1(u)-\varrho_2(u)\|_{C(\overbar B_\delta(z))} \sqrt{|\omega_\eps|}\|\varphi\|_{L_{2}(\Dsf)}
      \end{split}
	\een
	and
	\ben
	|\rom{3}(\eps,\varphi)| \le \|f_1-f_2\|_{L_\infty(B_\delta(z))} \sqrt{|\omega_\eps|}\|\varphi\|_{L_{2}(\Dsf)}.
	\een
	Now testing \eqref{eq:diff_u_ueps} with $\varphi = u_\eps - u$ and using \eqref{eq:I_III} together with  Assumption~\ref{ass:varrho_monotone},(a)-(b) lead to the desired estimate.   
\end{proof}

\subsection{Analysis of the averaged adjoint equation}
We introduce for $\eps\in [0,\tau]$  the (not necessarily symmetric) bilinear form $b_\eps:H^1_0(\Dsf) \times H^1_0(\Dsf)\to \VR$ by 
\ben
b_\eps(\psi, \varphi) := \int_{\Dsf} \beta_{\eps} \nabla \psi\cdot \nabla \varphi + \left(\int_0^1 \varrho_\eps'(su_\eps+(1-s)u)\;ds\right) \varphi \psi \;dx,
\een
where $\varrho_\eps'(u):= \varrho_1'(u)\chi_{\Omega_\eps} + \varrho_2'(u)\chi_{\VR^d\setminus \Omega_\eps}$
Then the averaged adjoint equation \eqref{eq:aa_equation} for the Lagrangian $G$ given by \eqref{eq:lagrange_scale_eps}  reads: find $q_\eps\in H^1_0(\Dsf)$ such that
\ben\label{eq:perturbed_adjoint_topo}
b_\eps(\psi,\fp^\eps)= - \int_{\Dsf} (u+u_\eps) \psi \;dx
\een
for all $\psi \in H^1_0(\Dsf)$. In view of Assumption~\ref{ass:varrho_monotone} we have $\varrho_\eps'\ge 0$ and $\beta_\eps\ge \beta_m I$ and thus $b_\eps$ is coercive,
\ben
b_\eps(\psi,\psi) \ge \beta_m \|\nabla \psi\|_{L_2(\Dsf)^d}^2 \quad \text{ for all } \psi\in H^1_0(\Dsf), \; \eps\in [0,\tau]. 
\een
As for the state equation, we use the notation $q:=q^0$. 
\begin{lemma}
  \begin{itemize}
    \item[(i)] For each $\eps\in[0,\tau]$ equation \eqref{eq:perturbed_adjoint_topo} admits a unique solution. 
  \item[(ii)] We find for every $z\in \Dsf\setminus \overbar \Omega$ 
    a number $\delta >0$, such that $q\in H^3(B_\delta(z))\subset C^1(\overbar B_\delta(z))$ for $d\in \{2,3\}$. 
  \end{itemize}
\end{lemma}
\begin{proof}
  (i) Since $b_\eps$ is coercive and continuous on $H^1_0(\Dsf)$, the theorem of Lax-Milgram shows that \eqref{eq:perturbed_adjoint_topo} admits a unique solution. 

(ii) The proof is the same as the one for item (ii) of Lemma~\ref{lem:aprior_u} and therefore omitted. 
\end{proof}

The previous lemma shows that $Y(\eps,u,u_\eps)=\{q_\eps\}$ is a singleton and therefore Hypothesis~(H2) is satisfied. We proceed with a H\"older-type estimate for $\eps \mapsto q_\eps$.

\begin{lemma}\label{lem:dif_p_peps}
	There is a constant $C>0$, such that for all small $\eps >0$,
	\ben\label{eq:p_pesp}
	\|q_\eps-q\|_{H^1(\Dsf)}\le C(\|u_\eps-u\|_{L_2(\Dsf)} + \eps^{d/2}).
	\een
\end{lemma}
\begin{proof}
Using \eqref{eq:perturbed_adjoint_topo} we obtain 
\ben\label{eq:beps_b}
\begin{split}
  b_\eps(\psi,\fp_\eps -\fp) &= b_\eps(\psi,\fp_\eps) - b_\eps(\psi,\fp)\\
			     & \stackrel{\eqref{eq:perturbed_adjoint_topo}}{=} -\int_\Dsf (u_\eps -u)\psi\;dx - (b_\eps-b_0)(\psi,\fp)
\end{split}
\een
for all $\psi \in H^1_0(\Dsf)$. Since furthermore
\ben
(b_\eps-b_0)(\psi,\fp)=-\int_{\omega_\eps} (\beta_1-\beta_2)\nabla \fp \cdot \nabla \psi\; dx   \\ 
- \int_{\omega_\eps} \left(\int_0^1(\varrho_1'-\varrho_2')(su_\eps +(1-s)u)\;ds \right) \fp \psi\;dx,
\een
we obtain using H\"older's inequality and $\fp \in C^1(\overbar B_\delta(z))$, 
\ben\label{eq:beps_b2}
\begin{split}
  |(b_\eps-b_0)(\psi,\fp)|\le &  \|\beta_1-\beta_2\|_{C(\overbar B_\delta(z))^{d\times d}} \|\nabla \fp\|_{C(\overbar B_\delta(z))^d} \sqrt{|\omega_\eps|} \|\nabla\psi\|_{L_{2}(\Dsf)^d} \\
			      & +  \|\varrho_1'-\varrho_2'\|_{L_\infty(B_C(0))} \|\fp\|_{L_2(\Dsf)} \|\psi\|_{L_{2}(\Dsf)}.
\end{split}
\een
where $C>0$ is a constant, such that $\|u_\eps\|_{L_\infty(\Dsf)}\le C$ for all $\eps\in [0,\tau].$ So inserting $\psi=\fp_\eps-\fp$ as test function in \eqref{eq:beps_b} and using \eqref{eq:beps_b2} yields
\ben
\beta_m \|\nabla(\fp_\eps - \fp)\|_{L_2(\Dsf)^d} \le b_\eps(\fp_\eps-\fp,\fp_\eps-\fp) \le C(\sqrt{\omega_\eps} +\|u_\eps-u\|_{L_2(\Dsf)})\|\fp_\eps-\fp\|_{H^1(\Dsf)}. 
\een
Now the result follows from the Poincar\'e inequality and $|\omega_\eps| = \eps^d |\omega|$.
\end{proof}

\begin{remark}
  The proof of estimate \eqref{eq:p_pesp} requires $q\in C^1(\overbar B_\delta(z))$, but not $q_\eps \in C^1(\overbar B_\delta(z))$, which is false in general, since  $\nabla q_\eps$ has a jump across $\partial \omega_\eps$. 
\end{remark}
	Let us finish this section with the verification of Hypothesis~(H3).
	\begin{lemma}\label{lem:hypo_H2}
We have 
	  \ben\label{eq:hypo_H2}
	  \begin{split}
	    \lim_{\eps \searrow 0}\frac{G(\eps, u,q) - G(0,u,q)}{|\omega_\eps|} = &  (\beta_1-\beta_2)(z) \nabla u(z)\cdot \nabla q(z) \\
									       & + (\varrho_1(u(z)) - \varrho_2(u(z))) q(z) \\
									       & - ((f_1-f_2)\fp)(z).
	\end{split}
	\een
	\end{lemma}
	\begin{proof}	
	   The change of variables $T_\eps$ shows that for $\eps >0$,
	\ben\label{eq:verify_H2}
	\begin{split}
	  \frac{G(\eps, u,\fp) - G(0,u,\fp)}{|\omega_\eps|}  = & \frac{1}{|\omega|}  \int_{\omega} ((\beta_1-\beta_2)\nabla u\cdot \nabla \fp)(T_\eps(x)) \; dx\\
							& + \frac{1}{|\omega|}  \int_{\omega} \left((\varrho_1(u)-\varrho_2(u)) u \fp\right)(T_\eps(x)) \; dx\\
							& - \frac{1}{|\omega|}  \int_{\omega} ((f_1-f_2)\fp)(T_\eps(x))  \; dx.
      \end{split}
	\een
	Recalling that $f_1,f_2\in C(\overbar B_\delta(z))$ and $u,\fp\in C^1(\overbar B_\delta(z))$ for a small $\delta>0$ and since $T_\eps(\omega)\subset \overbar B_\delta(z)$ for all small $\eps>0$, we can pass to the limit in \eqref{eq:verify_H2} to obtain \eqref{eq:hypo_H2}.
      \end{proof}

\subsection{Variation of the averaged adjoint equation and its weak limit}
The goal of this section is to verify Hypothesis~(H4), that is, to show that
\ben
R(u,\fp) := \lim_{\eps\searrow 0} \frac{G(\eps,u,\fp_\eps)-G(\eps,u,\fp)}{|\omega_\eps|}
\een
exists and, if possible, to determine its explicit form. In contrast to previous works we consider the variation of the 
averaged adjoint state variable which we will show converges weakly to a function $\fP$ defined on the whole space $\VR^d$. For this purpose we need the following definition.  
	\begin{definition}
	  The \emph{inflation} of $\Dsf\setminus \overbar\Omega$ around $z\in \Dsf\setminus \overbar\Omega$ is defined by $\Dsf_\eps := T_\eps^{-1}(\Dsf\setminus \overbar\Omega)$, where the transformation $T_\eps$ is defined by $T_\eps(x) := \eps x + z$.
      \end{definition}
      Notice that $\cup_{\eps >0} \Dsf_\eps = \VR^d$ and that $\eps \mapsto  \Dsf_\eps$ is monotonically decreasing, that is,  $\eps_1<\eps_2 \Rightarrow  D_{\eps_2} \subset D_{\eps_1} $.  
      \begin{lemma}\label{lem:trafo_eps}
	For $\eps >0$ we have $\varphi \in H^1(\Dsf\setminus\overbar{\Omega})$ if and only if $\varphi\circ T_\eps \in H^1(\Dsf_\eps)$. 
      \end{lemma}
      \begin{proof}
	Since $T_\eps$ is bi-Lipschitz continuous for $\eps >0$, this follows from \cite[Thm. 2.2.2, p.52]{b_ZI_1989a}.
      \end{proof}

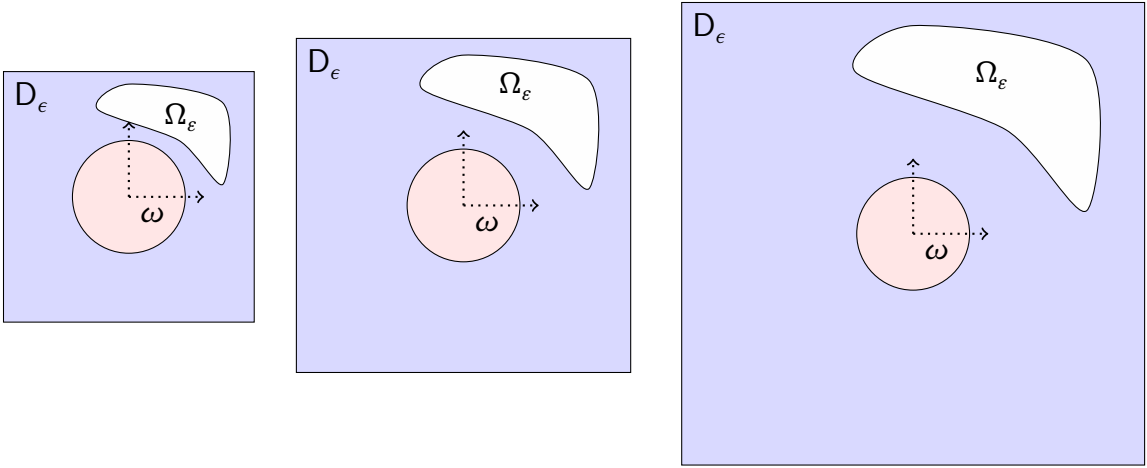
\begin{figure}[t]
\begin{minipage}[t][5cm][b]{.24\textwidth}
\raisebox{0.2\height}{
\begin{tikzpicture}

\newcommand{\rr}{1.2}

\coordinate (O) at (0,0);

\coordinate (A) at (-2,2);
\coordinate (B) at (-2,-2);
\coordinate (C) at (2,-2);
\coordinate (D) at (2,2);

\coordinate (C1) at (-0.5,1.4);
\coordinate (C2) at (0.0,1.8);
\coordinate (C3) at (1.5,1.5);
\coordinate (C4) at (1.5,0.2);
\coordinate (C5) at (0.8,0.9);

\draw[fill=blue!15] ($(O)+{(1/\rr)}*(A)$) -- ($(O)+{(1/\rr)}*(B)$) -- ($(O)+{(1/\rr)}*(C)$) -- ($(O)+{(1/\rr)}*(D)$)   -- cycle node[anchor=north west]{$\Dsf_\epsilon$};

\draw[fill=red!10] ($(O)$) circle ( 0.75 cm) node[anchor=north west]{$\omega$};

\draw [fill=gray!1] plot [smooth cycle] coordinates { ($(O)+{(1/\rr)}*(C1)$) ($(O)+{(1/\rr)}*(C2)$) ($(O)+{(1/\rr)}*(C3)$) ($(O)+{(1/\rr)}*(C4)$) ($(O)+{(1/\rr)}*(C5)$)}; 

\node at (20pt,30pt) {$\Omega_\eps$};

\draw[thick,->,dotted] (O) -- ($(O) + (1.0,0)$); 
\draw[thick,->, dotted] (O) -- ($(O) + (0,1.0)$); 

\end{tikzpicture}
}
\end{minipage}%
\begin{minipage}[t][5cm][b]{.3\textwidth}
\begin{tikzpicture}

\newcommand{\rr}{0.9}

\coordinate (O) at (0,0);

\coordinate (A) at (-2,2);
\coordinate (B) at (-2,-2);
\coordinate (C) at (2,-2);
\coordinate (D) at (2,2);

\coordinate (C1) at (-0.5,1.4);
\coordinate (C2) at (0.0,1.8);
\coordinate (C3) at (1.5,1.5);
\coordinate (C4) at (1.5,0.2);
\coordinate (C5) at (0.8,0.9);

\draw[fill=blue!15] ($(O)+{(1/\rr)}*(A)$) -- ($(O)+{(1/\rr)}*(B)$) -- ($(O)+{(1/\rr)}*(C)$) -- ($(O)+{(1/\rr)}*(D)$)   -- cycle node[anchor=north west]{$\Dsf_\epsilon$};

\draw[fill=red!10] ($(O)$) circle ( 0.75 cm) node[anchor=north west]{$\omega$};

\draw [fill=gray!1] plot [smooth cycle] coordinates { ($(O)+{(1/\rr)}*(C1)$) ($(O)+{(1/\rr)}*(C2)$) ($(O)+{(1/\rr)}*(C3)$) ($(O)+{(1/\rr)}*(C4)$) ($(O)+{(1/\rr)}*(C5)$)}; 

\node at (20pt,45pt) {$\Omega_\eps$};

\draw[thick,->,dotted] (O) -- ($(O) + (1.0,0)$); 
\draw[thick,->, dotted] (O) -- ($(O) + (0,1.0)$); 
\end{tikzpicture}
\end{minipage}%
\begin{minipage}[t][5cm][b]{.4\textwidth}
\raisebox{-0.2\height}{
\begin{tikzpicture}

\newcommand{\rr}{0.65}

\coordinate (O) at (0,0);

\coordinate (A) at (-2,2);
\coordinate (B) at (-2,-2);
\coordinate (C) at (2,-2);
\coordinate (D) at (2,2);

\coordinate (C1) at (-0.5,1.4);
\coordinate (C2) at (0.0,1.8);
\coordinate (C3) at (1.5,1.5);
\coordinate (C4) at (1.5,0.2);
\coordinate (C5) at (0.8,0.9);

\coordinate (Z) at (-6,10);

\draw[fill=blue!15] ($(O)+{(1/\rr)}*(A)$) -- ($(O)+{(1/\rr)}*(B)$) -- ($(O)+{(1/\rr)}*(C)$) -- ($(O)+{(1/\rr)}*(D)$)   -- cycle node[anchor=north west]{$\Dsf_\epsilon$};

\draw[fill=red!10] ($(O)$) circle ( 0.75 cm) node[anchor=north west]{$\omega$};

\draw [fill=gray!1] plot [smooth cycle] coordinates { ($(O)+{(1/\rr)}*(C1)$) ($(O)+{(1/\rr)}*(C2)$) ($(O)+{(1/\rr)}*(C3)$) ($(O)+{(1/\rr)}*(C4)$) ($(O)+{(1/\rr)}*(C5)$)}; 

\node at (30pt,60pt) {$\Omega_\eps$};

\draw[thick,->,dotted] (O) -- ($(O) + (1.0,0)$); 
\draw[thick,->, dotted] (O) -- ($(O) + (0,1.0)$); 

\end{tikzpicture}
}
\end{minipage}
\caption{Depicted are several inflated domains $\Dsf_\eps=T_\eps^{-1}(\Dsf\setminus \overbar \Omega)$ and $\Omega_\eps := T_\eps^{-1}(\Omega)$ with $\eps$ decreasing from left to right. The original inclusion $\omega_\eps$ appears as the fixed inclusion $\omega$ centered at the origin in the inflated domain. It can be seen that the domain $\Omega_\eps$ is gradually pushed to infinity the smaller $\eps$ gets. } 
\end{figure}

The next step is to consider the variation of the averaged adjoint state. For this purpose let us extend $\fp_\eps$ to zero outside of $\Dsf$, that is,  
\ben\label{eq:definition_p_tilde}
\tilde \fp_\eps(x) := \left\{
  \begin{array}{cl}
    \fp_\eps(x) & \text{ for a.e. } x\in \Dsf,\\
      0 & \text{ for a.e. } x\in \VR^d\setminus\Dsf.
\end{array}\right.
\een
In the same way we extend $u_\eps$ to a function $\tilde u_\eps:\VR^d\to \VR$. 
Notice that $\tilde u_\eps, \tilde \fp_\eps \in H^1(\VR^d)$ for all $\eps >0$. 
We will use the notation $\fp^\eps := \tilde\fp_\eps \circ T_\eps$. 

\begin{remark}[Neumann boundary conditions]
  If we had imposed Neumann conditions in (S), then it would be sufficient to replace \eqref{eq:definition_p_tilde} by $\tilde q_\eps := Eq_\eps$, where $E:H^1(\Dsf) \to H^1(\VR^d)$ is a continuous extension operator; see \cite[Thm. 1, pp. 254]{b_EV_1998a}. The subsequent analysis were still the same.
\end{remark}

\begin{definition}\label{def:variation_q}
	The \emph{variation of the averaged adjoint state} $\fp_\eps$ is defined pointwise a.e. in $\VR^d$ by
	\ben
	Q^\eps(x) := \frac{\tilde \fp_\eps(T_\eps(x))-\tilde \fp(T_\eps(x))}{\eps}.
	\een
	Notice that for every $\eps > 0$ we have  $\fP^\eps \in H^1(\VR^d)$.
	\end{definition}

	Our next task is to show that $(\fP^\eps)$ converges in $\dot{BL}(\VR^d)$ to a equivalence class of functions $[\fP]$ and determine 
	an equation for it. The first step is to prove the following apriori estimates.
	\begin{lemma}\label{lem:bounds_scale_one}
		There is a constant $C>0$, such that for all small $\eps >0$, 
		\ben\label{eq:bound_p_eps}
		\int_{\VR^d} (\eps \fP^\eps)^2+ |\nabla \fP^\eps|^2 \; dx \le C.
		\een
	\end{lemma}
	\begin{proof}
	  Obviously, the Lemmas~\ref{lem:dif_p_peps}/\ref{lem:u_ueps} imply that there is a constant $C>0$ such that $\|\fp_\eps -\fp\|_{H^1(\Dsf)}\le C\eps^{d/2}$ for all small $\eps>0$. 	This and definition \eqref{eq:definition_p_tilde} imply 
	\ben\label{eq:bound_p_eps_proof}
	\int_{\VR^d} (\tilde \fp_\eps-\tilde \fp)^2 + |\nabla (\tilde \fp_\eps-\tilde \fp)|^2\; dx \le C\eps^d.
		\een
		Hence invoking the change of variables $T_\eps$ in \eqref{eq:bound_p_eps_proof} yields 
		the bound \eqref{eq:bound_p_eps}.
	\end{proof}
      Notice that for $\eps >0$ the function $\fP^\eps$ belongs to $H^1(\VR^d)$, but it is not bounded with respect to $\eps$. However, the bound \eqref{eq:bound_p_eps} is sufficient to show the following key theorem.
      
      \begin{theorem}\label{thm:convegence_peps}
	For $d\in \{2,3\}$, we have
	\ben
	\begin{split}
	  \nabla \fP^\eps \rightharpoonup \nabla \fP \quad & \text{ in } L_2(\VR^d)^d,\\
	  \eps \fP^\eps \rightharpoonup 0 \quad & \text{ in } H^1(\VR^d),
      \end{split}
	\een
	where  $[\fP]\in \dot{BL}(\VR^d)$ is the unique solution to 
\ben\label{eq:exterior_beta}
\int_{\VR^d} A \nabla \psi\cdot \nabla \fP  \; dx = \int_\omega \zeta \cdot \nabla \psi\; dx \quad \text{ for all } \psi \in BL(\VR^d),
\een
where  $A := \beta_1(z)\chi_\omega + \beta_2(z)\chi_{\VR^d\setminus \omega}$ and $\zeta := -(\beta_1(z)-\beta_2(z))\nabla \fp(z)$; see \eqref{eq:laplace_unbounded_intro}.
			\end{theorem}
	\begin{proof}  
	  Fix $\bar \eps >0$ and let $0<\eps <\bar \eps$. We first notice that using \eqref{eq:perturbed_adjoint_topo} we have
	  \ben\label{eq:b_0}
	  b_\eps(\psi,\fp_\eps -\fp) = - \int_{\Dsf}(u_\eps-u)\psi\;dx - (b_\eps - b_0)(\psi,\fp)
	  \een
	  for all $\psi \in H^1_0(\Dsf)$. The idea is now to choose appropriate test functions in \eqref{eq:b_0} and then pass to the limit. For this purpose let $\bar\psi\in H^1_0(\Dsf_{\bar \eps})$ be arbitrary and define $\psi:= \eps \bar\psi\circ T_\eps^{-1}$. 
	  Thanks to Lemma~\ref{lem:trafo_eps} we have $\psi \in H^1_0(\Dsf\setminus \overbar \Omega)$ and the latter space embeds via \eqref{eq:definition_p_tilde}
 into $H^1_0(\Dsf)$. Hence we readily check that for such a test function, using a change of variables, we have
 \begin{subequations}
	  \begin{align}
	    \label{eq:b_first}    b_\eps(\psi,\fp_\eps -\fp) = & \eps^d\int_{\Dsf_{\bar\eps}} A_\eps \nabla \bar \psi \cdot \nabla \fP^\eps\;dx  + \eps^{d+1}\underbrace{\int_{\Dsf_{\bar\eps}} (\varrho_\eps'(u) \circ T_\eps) \eps\fP^\eps \bar\psi\;dx}_{=: \rom{1}(\eps,\bar\psi)} \\ 
	    \label{eq:b_second}  (b_\eps - b_0)(\psi,\fp) = &  \eps^d \underbrace{\int_{\omega} (\beta_1-\beta_2)(T_\eps(x))\nabla \bar \psi \cdot\nabla \fp(T_\eps(x))\;dx}_{=:\rom{2}(\eps,\bar\psi)}   
	     \\ 
	  \nonumber
	  & + \eps^{d+1}\underbrace{\int_{\Dsf_{\bar \eps}} \left(\int_0^1 (\varrho_\eps'(su_\eps + (1-s)u)\circ T_\eps-\varrho_0'(u)\circ T_\eps\;ds\right)  \fp(T_\eps(x)) \bar\psi \;dx}_{=:\rom{3}(\eps,\bar\psi)}\\     \label{eq:b_third}\int_{\Dsf}(u_\eps-u)\psi\;dx =& \eps^{d+1} \underbrace{\int_{\Dsf_{\bar\eps}}(u_\eps\circ T_\eps -u\circ T_\eps)\bar \psi\;dx}_{=:\rom{4}(\eps,\bar\psi)}, 
	  \end{align}
\end{subequations}
where $A_\eps(x) := \beta_1(T_\eps(x))\chi_\omega(x) + \beta_2(T_\eps(x))\chi_{\VR^d\setminus \omega}$.
Therefore inserting \eqref{eq:b_first}-\eqref{eq:b_third} into \eqref{eq:b_0} we obtain
\ben\label{eq:Qeps_final}
\begin{split}
  \int_{\Dsf_{\bar\eps}} & A_\eps\nabla \bar\psi \cdot \nabla \fP^\eps\; dx   + \int_{\omega} (\beta_1-\beta_2)(T_\eps(x))\nabla \bar\psi \cdot\nabla \fp(T_\eps(x))\;dx= - \eps(\rom{1} - \rom{2} - \rom{3} + \rom{4})(\eps,\bar\psi) 
\end{split}
\een
for all $\eps <\bar \eps$ and all $\bar \psi \in H^1_0(\Dsf_{\bar\eps})$. The next step is to show that \rom{1}-\rom{4} are bounded. Using the boundedness of $u_\eps$ on $\Dsf$ we see that $\varrho_\eps'(s\tilde u_\eps  +(1-s)\tilde u)\circ T_\eps$ and $\varrho_0'(\tilde u)\circ T_\eps$ are bounded (independently of $\eps$) on $\VR^d$, too. Therefore H\"older's inequality yields
		\begin{align}
		  |\rom{1}(\eps,\bar\psi)| &\le c\|\eps{\fP}^\eps\|_{L_2(\VR^d)}\|\bar\psi\|_{L_2(\VR^d)} \stackrel{\eqref{eq:bound_p_eps}}{\le} C\|\bar\psi\|_{L_2(\VR^d)}, \label{eq:bound_A}\\
		  |\rom{2}(\eps,\bar\psi)| & \le c \|\nabla \fp\|_{C(\overbar B_\delta(z))}\|\nabla \bar\psi\|_{L_2(\VR^d)^d} \stackrel{\eqref{eq:bound_p_eps}}{\le} C\|\nabla\bar\psi\|_{L_2(\VR^d)^d},\\ 
		  |\rom{3}(\eps,\bar\psi)| & \le c\|q\|_{C(\overbar B_\delta(z))}\|\bar\psi\|_{L_2(\VR^d)}  \label{eq:bound_C} \\
                 |\rom{4}(\eps,\bar\psi)| & \le c\|\tilde u_\eps\circ T_\eps - \tilde u\circ T_\eps\|_{L_2(\VR^d)}\|\bar\psi\|_{L_2(\VR^d)} \stackrel{\eqref{eq:est_u_eps_D}}{\le} C\|\bar\psi\|_{L_2(\VR^d)} \label{eq:bound_D} 
	      \end{align}
	  for all $\bar \psi \in H^1_0(\Dsf_{\bar\eps})$ and $\eps \in [0,\bar\eps]$. 
	  Thanks to Lemma~\ref{lem:bounds_scale_one} the family $(Q^\eps)$ is bounded in $\dot{BL}(\VR^d)$. The latter space is a Hilbert space and therefore for every null-sequence $(\eps_n)$ we find a subsequence $(\eps_{n_k})$ and $[\fP]\in \dot{BL}(\VR^d)$, such that $\nabla Q^{\eps_{n_k}} \rightharpoonup \nabla \fP$ in $L_2(\VR^d)^d$, where $\fP\in [\fP]$. Hence selecting $\eps = \eps_{n_k}$ in \eqref{eq:Qeps_final} and taking into account \eqref{eq:bound_A}-\eqref{eq:bound_C} we can pass to the limit $k\to \infty$ and obtain
	  \ben\label{eq:weak_limit_Q}
	  \int_{\Dsf_{\bar \eps}} A \nabla \bar\psi \cdot \nabla \fP\;dx = - (\beta_1(z)-\beta_2(z))\nabla q(z)\cdot  \int_{\omega} \nabla \bar\psi \;dx \quad \text{ for all } \bar\psi \in H^1_0(\Dsf_{\bar \eps}). 
	      \een
  The mapping $\bar\eps \mapsto \Dsf_{\bar\eps}$ is monotonically decreasing and we have $H^1_0(\Dsf_{\bar \eps})\subset H^1_0(\VR^d)$. This shows, recalling that $\bar \eps>0$ is arbitrary, that $\Dsf_{\bar \eps}$ appearing in \eqref{eq:weak_limit_Q} may be replaced by $\VR^d$. But this means that $\fP$ is the unique solution of \eqref{eq:exterior_beta}.

  Let us now show that $\eps \fP^\eps \rightharpoonup 0$ in $H^1(\VR^d)$ as $\eps \searrow 0$. From the first part of the proof it is clear that $\nabla (\eps \fP^\eps) \rightharpoonup 0$ in $L_2(\VR^d)^d$. To see the weak convergence of $(\eps \fP^\eps)$ in $L_2(\VR^d)$ we fix $r>0$. Then Poincar\'e's inequality for a ball yields
	      \ben\label{eq:poincare}
	      \|(\eps \fP^\eps)_r - \eps \fP^\eps \|_{L_2(B_r(0))} \le  \eps C(r) \|\nabla \fP^\eps \|_{L_2(B_r(0))^d},
	      \een
	      where $(\eps \fP^\eps)_{r} := \fint_{B_r(0)} \eps \fP^\eps\;dx$ denotes the average over the ball $B_r(0)$. Since the gradient $\|\nabla \fP^\eps\|_{L_2(\VR^d)^d}$ is uniformly bounded (see Lemma~\ref{lem:bounds_scale_one}), the right hand side of \eqref{eq:poincare} goes to zero as $\eps \searrow 0$. But also $\eps \fP^\eps$ is bounded in $L_2(\VR^d)$ and therefore we find for any null-sequence $(\eps_n)$ a subsequence $(\eps_{n_k})$ and $\hat{\fP} \in L_2(\VR^d)$, such that $\eps_{n_k} \fP^{\eps_{n_k}} \rightharpoonup \hat{\fP}$ in $L_2(\VR^d)$. It is clear that $(\eps_{n_k} \fP^{\eps_{n_k}})_{B_r(0)} \to (\hat{\fP})_{B_r(0)}$ in $\VR$. Therefore we obtain from \eqref{eq:poincare} together with the weak lower semi-continuity of the $L_2$-norm
	      \ben
	      \|(\hat{\fP})_r - \hat{\fP}\|_{L_2(B_r(0))}	 \le     \liminf_{k\to \infty} \|(\fP^{\eps_{n_k}})_{B_r(0)} - \fP^{\eps_{n_k}}\|_{L_2(B_r(0))} \le 0.
\een	      
This shows that $\hat{\fP} = (\hat{\fP})_{r}$ a.e. on $B_r(0)$ and thus $\hat{\fP}$ is constant on $B_r(0)$. 
Since $r>0$ was arbitrary, $\hat{\fP}$ must be constant on $\VR^d$. Further $\hat{\fP}\in L_2(\VR^d)$ implies $\hat{\fP}=0$ and this finishes the proof. 
\end{proof}

We are now ready to compute  $R(u,\fp)$ and thereby verify the second part of Hypothesis (H2). 
\begin{lemma}\label{lem:verify_H3}
We have 
\ben\label{eq:form_R}
R(u,\fp) = (\beta_1(z)-\beta_2(z)) \nabla u(z)\cdot \fint_{\omega} \nabla \fP\; dx,
\een
where $[\fP]$ is the solution to \eqref{eq:exterior_beta}.
\end{lemma}
\begin{proof}
  Testing the state equation \eqref{eq:PDE_constraint_topo} (for $\eps =0$) with $\varphi = \fp_\eps - \fp$ gives
  \ben\label{eq:state_test_qeps}
  \int_\Dsf \beta_0 \nabla u \cdot \nabla (\fp_\eps-\fp) + \varrho_0(u) (\fp_\eps -\fp) \;dx = \int_\Dsf f_0 (\fp_\eps -\fp) \;dx.  
  \een
Therefore we can write for $\eps >0$,  
\ben\label{eq:difference_G}
	\begin{split}
	  G(\eps,u,\fp_\eps)-G(\eps,u,\fp) = &  \int_{\Dsf} \beta_\eps \nabla u \cdot \nabla (\fp_\eps -\fp) + \varrho_\eps(u) (\fp_\eps-\fp) \; dx - \int_\Dsf f_\eps (\fp_\eps - \fp)\;dx \\
	  \stackrel{\eqref{eq:state_test_qeps}}{=} & \int_{\omega_\eps} (\beta_1-\beta_2) \nabla u \cdot \nabla (\fp_\eps - \fp) + [(\varrho_1-\varrho_2)(u)  - (f_1-f_2)] (\fp_\eps - \fp)\; dx.
	\end{split}
	\een
	 Invoking the change of variables $T_\eps$ in \eqref{eq:difference_G} we obtain for $\eps >0$
	\ben
	\begin{split}
	  \frac{G(\eps,u,\fp_\eps)-G(\eps,u,\fp)}{|\omega_\eps|} = & \frac{1}{|\omega|}\int_{\omega} [(\varrho_1-\varrho_2)(u(T_\eps(x))) - (f_1-f_2)(T_\eps(x))] \eps\fP^\eps \; dx \\
							      & +\frac{1}{|\omega|}\int_{\omega}  ((\beta_1-\beta_2) \nabla u)(T_\eps(x)) \cdot \nabla \fP^\eps \; dx   \\
          & \quad \to (\beta_1(z)-\beta_2(z)) \nabla u(z) \cdot \fint_{\omega} \nabla \fP \;dx, 
      \end{split}
	\een
	where in the last step we used Theorem~\ref{thm:convegence_peps}, $f_1,f_2\in C(\overbar B_\delta(z))$, and $u \in C^1(\overbar B_\delta(z))$ for $\delta >0$ small. 
\end{proof}

\subsection{Topological derivative and polarisation matrix}\label{subsec:topo}
\paragraph{Topological derivative}
Now we are in a position to formulate our main result. In the previous sections we have checked that Hypotheses (H0)-(H4) of Theorem~\ref{thm:diff_lagrange} 
are satisfied for the Lagrangian $G$ given by \eqref{eq:lagrange_scale_eps}. Therefore Theorem~\ref{thm:main} can be applied and we obtain the following result. 
\begin{theorem}\label{thm:main}
  The topological derivative of $J$ at $\Omega$ in $z\in \Dsf\setminus\overbar \Omega$ is given by 
  \ben\label{eq:topo_final}
  \lim_{\eps\searrow 0} \frac{J(\Omega\cup \omega_\eps(z))-J(\Omega)}{|\omega_\eps(z)|} =  \partial_\ell G(0, u,\fp) + R(u,\fp),
\een
where 
\ben
\partial_\ell G(0, u,\fp) = \big((\beta_1-\beta_2)\nabla u\cdot \nabla \fp + (\varrho_1(u) - \varrho_2(u)) \fp - (f_1-f_2)\fp\big)(z)
\een
and 
\ben\label{eq:form_R_main}
R(u,\fp) = (\beta_1(z)-\beta_2(z)) \nabla u(z)\cdot\fint_{\omega} \nabla \fP\; dx,
\een
where $\fP$ depends on $z$ and is the solution to \eqref{eq:exterior_beta}.
\end{theorem}

Next we rewrite the term $R(u,\fp)$ with the help of the so-called polarisation matrix. For this purpose we fix $z\in \Dsf\setminus\overbar \Omega$ in the following and denote by $[\fP_{\zeta}]$, $\zeta\in \VR^d$, the solution to \eqref{eq:exterior_beta0} with $A:=A_\omega := \beta_1(z)\chi_\omega + \beta_2(z)\chi_{\VR^d\setminus \omega}$. Also we denote by $\fP_{\zeta}$ an arbitrary representative of $[\fP_\zeta]$.

\begin{definition}
The matrix representing the linear \emph{averaging operator}
	\ben
	\zeta \mapsto \fint_{\omega} \nabla \fP_{\zeta}\; dx, \; \VR^d \mapsto \VR^d
	\een
	 is called \emph{weak polarisation matrix} and will be denoted $\Cp_z\in \VR^{d\times d}$. Notice that this matrix depends on $\beta_1(z)$ and $\beta_2(z)$. 
\end{definition}
We use the term weak polarisation matrix here, because it is defined via the weak formulation \eqref{eq:exterior_beta0} and therefore does not require any regularity assumptions on $\partial \omega$ or $\Omega$. We give another definition of a polarisation matrix later and relate it to the weak polarisation matrix. We also refer to \cite{a_POSZ_1951a} and the monograph \cite[Sec. 9.4.4, pp. 273]{b_NOSO_2013a}. 

	\begin{corollary}
       We have
       \ben
       \begin{split}
	 \lim_{\eps\searrow 0} \frac{J(\Omega\cup \omega_\eps(z))-J(\Omega)}{|\omega_\eps|}       = & (((\beta_1-\beta_2) \nabla u)\cdot (I-\Cp_z(\beta_1-\beta_2))\nabla \fp)(z)  \\ 
												    &+  ((\varrho_1(u) - \varrho_2(u)) \fp -(f_1-f_2)\fp)(z).
     \end{split}
       \een
	\end{corollary}
      \begin{proof}
	This follows at once from \eqref{eq:topo_final} noting that $\Cp_z\zeta = \fint_{\omega} \nabla \fP_{\zeta}\;dx$, where $\zeta := -(\beta_1(z)-\beta_2(z))\nabla \fp(z)$. 
      \end{proof}

\paragraph{Further properties of the polarisation matrix}
Next we derive further properties of the polarisation matrix. Furthermore
we relate our polarisation matrix to previous definitions. We refer the reader to \cite{b_AMKA_2007a} for further information on polarisation matrices.

\begin{lemma}\label{lem:pol_sym}
	If $\beta_2(z)=\beta_2^\top(z)$ and $\beta_1(z) = \beta_1^\top(z)$, then the polarisation matrix is symmetric, that is, $\Cp_z =\Cp_z^\top.$
      \end{lemma}
      \begin{proof}
  We compute for the $(i,j)$-entry of the polarisation matrix: 
  \ben\label{eq:polarisation_symm}
    \begin{split}
      e_i\cdot \Cp_ze_j  = e_i \cdot \fint_{\omega} \nabla \fP_{e_j}\;dx & \stackrel{\eqref{eq:exterior_beta}}{=} \int_{\VR^d} A_\omega \nabla \fP_{e_j}\cdot \nabla \fP_{e_i}  \; dx\\
								       & \stackrel{\text{sym. of } A_\omega}{=}  \int_{\VR^d} \nabla \fP_{e_j}\cdot A_\omega \nabla \fP_{e_i}  \; dx\\
		      & \stackrel{\eqref{eq:exterior_beta}}{=} e_j \cdot \fint_{\omega} \nabla \fP_{e_i}\;dx  = e_j\cdot \Cp_ze_i.
    \end{split}
    \een
This shows the symmetry. 
      \end{proof}
The polarisation matrix is also positive definite (even in the nonsymmetric case).
\begin{lemma}\label{lem:pol_pos}
The matrix $\Cp_z$ is positive definite. 
\end{lemma}
\begin{proof}
  Let $w=(w_1,\ldots, w_d)\in \VR^d$ be an arbitrary vector. Set $W := \sum_{i=1}^d w_i \fP_{e_i}$. Then we readily check using \eqref{eq:polarisation_symm},
  \ben\label{eq:wPw}
  w\cdot \Cp_zw = \int_{\VR^d} A_\omega \nabla W\cdot \nabla W\;dx \ge \beta_m \int_{\VR^d} |\nabla W|^2\;dx.
  \een
  This shows that $\Cp_z$ is positive semidefinite. Suppose now $w$ is such that $w\cdot \Cp_zw=0$. Then, in view of \eqref{eq:wPw}, we must have $[W]=[0]$. Hence \eqref{eq:exterior_beta} gives
  \ben\label{eq:w_dot_F}
   w\cdot \fint_{\omega}\nabla \varphi\;dx =0 \quad \text{ for all } \varphi \in \text{BL}(\VR^d).
  \een
  Let $V\subset \VR^d$ be a bounded and open set, such that $\omega\Subset V$. Choose a smooth function $\rho$, such that $\rho=1$ on $\omega$, $0\le \rho\le 1$ on $V\setminus \omega$ and $\rho=0$ outside of $V$. Then we define $\varphi(x) := e_i\cdot x \rho(x)$ for $i\in \{1,\ldots, d\}$, which belongs to $\text{BL}(\VR^d)$. Hence we may test \eqref{eq:w_dot_F} with this function and conclude $w_i=0$. This shows $w=0$ and finishes the proof.
\end{proof}

Suppose from now on $\beta_1=\gamma_1I$ and $\beta_2=\gamma_2I$ for $\gamma_1,\gamma_2>0$. We select $\fP_\zeta\in [\fP_\zeta]$ and suppose that it can be represented by a single layer potential: there is a function $h_\zeta\in C(\partial \omega)$, such that
\ben\label{eq:single_layer}
\fP_\zeta(x) = \int_{\partial \omega} h_\zeta(y) E(x-y)\; ds(y), \qquad \int_{\partial \omega} h_{\zeta}\;ds =0,
\een
where $E$ denotes the fundamental solution of $u\mapsto -\Delta u$; \cite[Chap. 3]{b_FO_1995a}. It is readily checked using \eqref{eq:single_layer} that $|\fP_\zeta(x)| = O(|x|^{1-d})$.

\begin{definition}
  The  \emph{strong polarisation matrix} is the matrix $\tilde{\Cp}_z =(\tilde{\Cp_z})_{ij} \in \VR^{d\times d}$ with entries
\ben
(\tilde{\Cp_z})_{ij} =  \int_{\partial \omega} x_j h_{e_i} \;ds.
\een
\end{definition}

The strong and weak polarisation matrices are related as shown in the following lemma.

\begin{lemma}
 Assume that $\partial \omega$ is $C^2$. Then we have 
\ben\label{eq:pol_weak-strong_}
\Cp_z =  -  \frac{1}{|\omega|}\frac{\beta_2}{\beta_1-\beta_2}\tilde{\Cp}_z + \frac{1}{\beta_1-\beta_2}I.
\een
\end{lemma}
\begin{proof}
At first we obtain by partial integration, noting that $e_i = \nabla x_i$,
\ben\label{eq:pol_strong-weak}
\begin{split}
  e_i\cdot \Cp_z e_j & = \fint_{\omega} \nabla x_i \cdot \nabla \fP_{e_j}\;ds = \frac{1}{|\omega|}\int_{\partial\omega} x_i \partial_\nu \fP_{e_j}\;ds -  \fint_{\omega} \underbrace{ \Delta \fP_{e_j}}_{=0, \text{ in view of } \eqref{eq:exterior_beta}}\;dx.
\end{split}
\een
Next we express $\partial_\nu \fP_{e_j}$ in terms of $h_{e_j}$. For this recall (see, e.g., \cite{b_FO_1995a}) that the jump condition
\ben\label{eq:jump_potential}
\begin{split}
  \partial_\nu \fP^+_{e_i} - \partial_\nu \fP^-_{e_i} = h_{e_i} \quad \text{ on } \partial \omega
\end{split}
\een
is satisfied. In addition we get from \eqref{eq:exterior_beta},
\ben\label{eq:jump_PDE}
\beta_1\partial_\nu \fP^+_{e_i} - \beta_2\partial_\nu \fP^-_{e_i}  = e_i\cdot \nu \quad \text{ on } \partial \omega.
\een
Combining \eqref{eq:jump_potential} and \eqref{eq:jump_PDE} we obtain
\ben
\partial_\nu \fP^+_{e_i} = - \frac{\beta_2}{\beta_1-\beta_2}h_{e_i} + \frac{1}{\beta_1-\beta_2}e_i\cdot \nu. 
\een
Inserting this expression into \eqref{eq:pol_strong-weak} yields
\ben
e_i\cdot \Cp_z e_j  = -  \frac{\beta_2}{\beta_1-\beta_2}\frac{1}{|\omega|}\int_{\partial\omega} x_i h_{e_j}\;ds + \frac{1}{\beta_1-\beta_2}\frac{1}{|\omega|}\int_{\partial\omega} (e_i \cdot \nu) x_j\;ds.
\een
This is equivalent to formula \eqref{eq:pol_weak-strong_}, since by Gauss's divergence theorem 
\ben
\frac{1}{|\omega|}\int_{\partial\omega} (e_i \cdot \nu) x_j \;ds= \fint_{\omega} \underbrace{\Div(e_ix_j)}_{=\delta_{ij}}\;dx = \delta_{ij}.
\een
\end{proof}


\begin{remark}
	In some cases, see, e.g., \cite{a_AM_2006a,a_KIKAKI_2003a, a_AMKAKI_2005a}, the polarisation matrix can be computed explicitly: for instance when $\beta_1=\gamma_1 I,\beta_2=\gamma_2I$, $\beta_1,\beta_2>0$, and $\omega$ is a circle or more generally an ellipse. However for general inclusions $\omega$ the exterior equation \eqref{eq:exterior_beta} has to be solved numerically in order to evaluate formula \eqref{eq:topo_final}.
\end{remark}

\section{The extremal case of void material}
In this last section we discuss the extremal situation in which $\beta_1=0$, $\varrho_1=0$ and $f_1=0$ in \eqref{eq:state_weak_chi}. This case corresponds to the insertion of a hole with Neumann boundary conditions imposed on the inclusion; see \cite{a_IGNAROSOSZ_2009a}. Since the extremal case is similar to the considerations from the previous section, we will only work out the main differences in detail.

\subsection{Problems setting}
We suppose as before that $\Dsf\subset \VR^d$ is a bounded Lipschitz domain. For a simply connected domain $\Omega\Subset \Dsf$ with Lipschitz boundary $\partial \Omega$, we consider the shape function
\ben\label{eq:cost_extreme}
J(\Omega) = \int_{\Dsf\setminus \overbar \Omega} u^2\;dx
\een
subject to $u=u_\Omega\in H^1_{\partial \Dsf}(\Dsf\setminus \overbar \Omega)$ solves
\ben\label{eq:state_weak_chi_extreme}
\int_{\Dsf\setminus \overbar \Omega} \beta_2 \nabla u\cdot \nabla \varphi + \varrho_2(u)\varphi \; dx = \int_{\Dsf\setminus \overbar \Omega} f_2\varphi\;dx \quad \text{ for all }  \varphi \in H^1_{\partial \Dsf}(\Dsf\setminus \overbar \Omega),
\een
where $H^1_{\partial \Dsf}(\Dsf\setminus \overbar \Omega):=\{v\in H^1(\Dsf\setminus \overbar \Omega):\;v=0 \text{ on } \partial \Dsf\}$.
This setting corresponds to the limiting case of \eqref{eq:state_weak_chi} in which $\beta_1=0$, $\varrho_1=0$ and $f_1=0$. 

The rest of this section is dedicated to the computation of the topological sensitivity of $J$ at $\Omega=\emptyset$ with respect to the inclusion $\omega$ (which will be specified below), i.e., 
\ben
\lim_{\eps\searrow 0}  \frac{J(\omega_\eps) - J(\emptyset)}{|\omega_\eps|}. 
\een
We will see that almost all steps are the same as in the last section with two main differences. The first main difference being that $X(\eps)$ is not 
a singleton and that we have to introduce a new equation on the inclusion, which requires 
a more detailed explanation and a thorough  analysis. The second difference concerns 
the required assumptions on the regularity of the inclusion $\omega$. While in the previous section it was sufficient to assume that $\omega$ is merely an open set, here we strengthen the assumption and assume that $\omega$ is a simply connected Lipschitz domain. 

\begin{assumption}\label{ass:varrho_monotone_extreme}
We assume that either 
\begin{itemize}
  \item[(a)] $\beta_2\in \VR^{d\times d}$ is symmetric, positive definite and $\varrho_2$ satisfies Assumption~\ref{ass:varrho_monotone}, (b) and it is bounded, or 
\end{itemize}
\begin{itemize}
  \item[(b)] $\beta_2$ satisfies Assumption~\eqref{ass:varrho_monotone}, (a) and $\varrho_2$ satisfies Assumption~\ref{ass:varrho_monotone}, (b) and additionally $\varrho_2'> \lambda$ for some $\lambda >0$, 
\end{itemize}
is satisfied. In both cases we assume $f_2\in H^1(\Dsf)\cap C(\Dsf)$. 
\end{assumption}
Under these assumptions we can prove, using similar arguments as in Lemma~\ref{lem:aprior_u}, that \eqref{eq:state_weak_chi_extreme} admits a unique solution and that there is a constant $C>0$ (depending on $\Omega$), such that 
\ben\label{eq:bound_u_extreme}
      \|u_\Omega\|_{L_\infty(\Dsf\setminus \overbar \Omega)} +  \|u_\Omega\|_{H^1_0(\Dsf\setminus \overbar \Omega)}\le C\|f_2\|_{L_r(\Dsf\setminus\overbar \Omega)}
\een
for $r>d/2$ close enough to $d/2$.  
Moreover, for every $z\in \Dsf\setminus \overbar \Omega$, we find $\delta>0$, such that $u_{\Omega}\in H^3(B_\delta(z))$.

\subsection{The parametrised Lagrangian}
From now on we fix:
\begin{itemize}\setlength\itemsep{0.3em}
  \item a simplify connected Lipschitz domain $\omega\subset \VR^d$ with $0\in \omega$,
  \item  a point $z\in \Dsf$,
  \item the perturbation $\Omega_\eps := \omega_\eps:=\omega_\eps(z)$, where $\omega_\eps(z) := z+\eps \omega$ and $\eps\in[0,\tau]$, $\tau >0$.
\end{itemize}
Let $X=Y=H^1_0(\Dsf)$ and introduce the Lagrangian $ G:[0,\tau]\times X\times Y \to \VR$ associated with the perturbation $\Omega_\eps$ by 
\ben\label{eq:lagrange_scale_eps_extreme}
G(\eps,\fu,\fp) := \int_{\Dsf\setminus\overbar \omega_\eps} \fu^2\;dx + \int_{\Dsf\setminus\overbar \omega_\eps} \beta_2 \nabla \fu \cdot \nabla\fp +  \varrho_2(\fu)\fp \;dx - \int_{\Dsf\setminus\overbar \omega_\eps} f_2\fp \;dx.
\een
We will verify that Hypotheses~(H0)-(H4) are satisfied with $\ell(\eps) = |\omega_\eps|$. 

\subsection{Analysis of the perturbed state equation}
The \emph{perturbed state equation} reads: find $u_\eps \in H^1_0(\Dsf)$ such that $\partial_p G(\eps, u_\eps, 0)(\varphi)=0$ for all $ \varphi \in H^1_0(\Dsf),$
or equivalently $u_\eps \in H^1_0(\Dsf)$ satisfies,
\ben\label{eq:PDE_constraint_topo_extreme}
\int_{\Dsf\setminus\overbar\omega_\eps} \beta_2 \nabla u_\eps\cdot \nabla\varphi + \varrho_2(u_\eps)\varphi \;dx = \int_{\Dsf\setminus\overbar\omega_\eps} f_2\varphi \;dx \quad \text{ for all } \varphi \in H^1_0(\Dsf).
\een
Henceforth we write $u:=u_0$ to simplify notation. Since \eqref{eq:state_weak_chi_extreme} admits a unique solution $\bar u_\eps$ for $\Omega = \omega_\eps$, which can be extended to $H^1_0(\Dsf)$,  \eqref{eq:PDE_constraint_topo_extreme} admits a solution, too, whose restriction to $\Dsf\setminus \overbar \Omega$ is unique. This means that 
 \ben
 E(\eps) = \{u\in H^1_0(\Dsf):\;  u = \bar u_\eps \text{ a.e. on } \Dsf\setminus\overbar \omega_\eps \},
 \een
 where $\bar u_\eps$ is the unique solution to \eqref{eq:state_weak_chi_extreme}. It also follows that $X(\eps)=E(\eps)$ since the Lagrangian only depends on 
 the restriction of functions to $D\setminus \overbar\omega_\eps$. Note that the set $X(0)$ is a singleton.  Moreover for all $\eps \in [0,\tau]$, 
 \ben
 g(\eps) = \inf_{u\in E(\eps)} G(\eps,u,0) = \int_{\Dsf\setminus \overbar\omega_\eps} \bar u^2_\eps \;dx. 
 \een
 This shows that Hypothesis (H0) and, in view of Assumption~\ref{ass:varrho_monotone_extreme}, also Hypothesis (H1) is satisfied. 

 The next step deviates from the transmission problem case (of Section~\ref{sec:4}). We construct functions $u_\eps \in X(\eps)$ and 
 $\fp_\eps \in Y(\eps,u_0, u_\eps)$ that satisfy Hypothesis~(H4). For this purpose we associate with $u_\eps \in H^1_{\partial \Dsf}(\Dsf\setminus \overbar \omega_\eps)$ a function $u_\eps^+\in H^1(\omega_\eps)$ defined as the unique weak solution to the Dirichlet problem
\ben\label{eq:extension}
\begin{split}
  -\Div(\beta_2 \nabla u_\eps^+)  +\varrho_2(u) & =f_2 \quad \text{ in } \omega_\eps \\
  u_\eps^+ & = u_\eps \quad \text{ on } \partial \omega_\eps. 
\end{split}
\een
With this function we can extend $u_\eps$ to a function $u_\eps\in H^1_0(\Dsf)$ by setting 
\ben
u_\eps := \left\{
  \begin{array}{cl}
    u_\eps^+ & \text{ in } \omega_\eps\\
    u_\eps   & \text{ in } \Dsf\setminus \overbar \omega_\eps
\end{array}\right..
\een
 Now we prove the following analogue of Lemma~\ref{lem:u_ueps}. 
\begin{lemma}\label{lem:u_ueps_extreme}
	There is a constant $C>0$, such that for all small $\eps>0$, 
	\ben\label{eq:est_u_eps_D_extreme}
	\|u_\eps - u\|_{H^1(\Dsf)} \le C\eps^{d/2}. 
	\een
\end{lemma}
\begin{proof}
We first establish an estimate for $u_\eps-u$ on $\omega_\eps$. For this purpose we fix a bounded domain $S\subset \Dsf$ containing $\omega$.  We note that the difference $e_\eps(x) := u_\eps(T_\eps(x)) - u(T_\eps(x))$ satisfies $-\Div(\beta_2\circ T_\eps\nabla e_\eps)=0$ on $\omega$ and 
$e_\eps = u_\eps(T_\eps(x)) - u(T_\eps(x))$ on $\partial \omega$. Hence by standard elliptic regularity and the trace theorem we find  
\ben\label{eq:e_lambda}
\| e_\eps+\lambda\|_{H^1(\omega)} \le  c\|e_\eps+\lambda\|_{H^{1/2}(\partial \omega)} \le c \|e_\eps + \lambda\|_{H^1(S\setminus \overbar \omega)}
\een
for all $\lambda\in \VR$. Since the quotient norms on the spaces $H^1(\omega)/\VR$ and $H^1(S\setminus \overbar\omega)$ are equivalent to the seminorms $|v|_{H^1(\omega)}:=\|\nabla v\|_{L_2(\omega)^d}$ and $|v|_{H^1(S\setminus\overbar\omega)}:=\|\nabla v\|_{L_2(S\setminus \overbar\omega)^d}$, respectively,  we conclude $\|\nabla e_\eps\|_{L_2(\omega)^d} \le c\|\nabla e_\eps\|_{L_2(S\setminus \omega)^d}$. Therefore estimating the right hand side and changing variables shows 
\ben\label{eq:est_in_out}
\|\nabla(u_\eps - u)\|_{L_2(\omega_\eps)^d} \le c \|\nabla(u_\eps-u)\|_{L_2(\Dsf\setminus \overbar \omega_\eps)^d}.
\een
A fortiori using \eqref{eq:est_in_out} and a similar argument shows that \eqref{eq:e_lambda} implies
\ben\label{eq:est_in_out2}
\|u_\eps - u\|_{L_2(\omega_\eps)} \le c(\eps \|\nabla(u_\eps-u)\|_{L_2(\Dsf\setminus \overbar \omega_\eps)^d} + \|u_\eps - u\|_{L_2(\Dsf\setminus \overbar \omega_\eps)}).
\een
This finishes the first step of the proof. We now establish an estimate for the right hand side of \eqref{eq:est_in_out}. Following the steps of Lemma~\ref{lem:u_ueps} we find 
  \ben\label{eq:diff_u_ueps_extreme}
	\begin{split}
	  \int_{\Dsf\setminus\overbar\omega_\eps} & \beta_2 \nabla (u_\eps-u)\cdot \nabla\varphi \;dx + \int_{\Dsf\setminus\overbar \omega_\eps} (\varrho_2(u_\eps) - \varrho_2(u)) \varphi\; dx  \\
	= &  \int_{\omega_\eps} \beta_2\nabla u\cdot \nabla \varphi\;dx + \int_{\omega_\eps} \varrho_2(u) \varphi -  f_2 \varphi\; dx
        \end{split}
	\een
	for all $\varphi \in H^1_0(\Dsf)$. Let us first assume that Assumption~\ref{ass:varrho_monotone_extreme}, (a) holds.  
	Fix $\bar \eps >0$ and let $0<\eps <\bar \eps$. Changing variables in  \eqref{eq:diff_u_ueps_extreme} yields (recalling that we denote by $\tilde u_\eps$ the extension of $u_\eps$ to $\VR^d$) 
	\ben\label{eq:diff_u_ueps_scale1_extreme}
	\begin{split}
	  \int_{\VR^d\setminus\overbar\omega} & \beta_2 \nabla K_\eps\cdot \nabla\varphi \;dx = -\underbrace{ \eps^2 \int_{\VR^d\setminus\overbar \omega} (\varrho_2(\tilde u_\eps(T_\eps)) - \varrho_2(\tilde u(T_\eps))) \varphi\; dx}_{\to 0, \text{ since } \varrho_2 \text{ is bounded}}  \\
	  + &  \underbrace{\eps\int_{\omega} \beta_2\nabla u(T_\eps)\cdot \nabla \varphi\;dx + \eps^2 \int_{\omega} \varrho_2(u(T_\eps) ) \varphi -  f_2(T_\eps) \varphi\; dx}_{\to 0, \text{ since } u\in C^1(\overbar B_\delta(z)), f_2\in C(\Dsf)\text{ and } \varrho_2\in C(\VR)},
        \end{split}
	\een
	for all $\varphi\in H^1_{\partial \Dsf}(\Dsf_{\bar\eps}\setminus\overbar \omega)$, where $K_\eps := (u_\eps-u)\circ T_\eps$. Since $\bar \eps >0$ is arbitrary, this shows that  
	$K_\eps \rightharpoonup 0$ weakly in $\dot{\text{BL}}(\VR^d\setminus \overbar \omega)$. But this means that $K_\eps$ must be bounded in $\dot{\text{BL}}(\VR^d\setminus \overbar \omega)$ and hence we find $C>0$, such that $\|\nabla K_\eps\|_{L_2(\VR^d\setminus \overbar \omega)^d}\le C$ or equivalently after changing variables $\|\nabla(u_\eps - u)\|_{L_2(\Dsf\setminus \overbar \omega_\eps)}\le C\eps^{d/2}$. Combining this estimate with \eqref{eq:est_in_out} and using Poincar\'e's inequality gives \eqref{eq:est_u_eps_D_extreme}. 
	
	Let us now assume that Assumption~\ref{ass:varrho_monotone_extreme}, (b) is satisfied. Testing \eqref{eq:diff_u_ueps_extreme} with $\varphi=u_\eps-u$, using $\varrho_2'>\lambda$ and applying H\"older's inequality yield
	\benn
	C\|u_\eps-u\|_{H^1(\Dsf\setminus\overbar\omega_\eps)}^2 \le \sqrt{|\omega_\eps|}(\|\nabla u\|_{C(B_\delta(z))^d}\|\nabla (u_\eps-u)\|_{L_2(\omega_\eps)^d} + \|\varrho(u)-f_2\|_{C(\overbar B_\delta(z))}\|u_\eps-u\|_{L_2(\omega_\eps)}).
	\eenn
	Using \eqref{eq:est_in_out} and \eqref{eq:est_in_out2} to estimate the right hand side and noting $|\omega_\eps|=|\omega|\eps^d$, we infer $\|u_\eps-u\|_{H^1(\Dsf\setminus\overbar\omega_\eps)} \le C\eps^{d/2}$. Again combining this estimate with \eqref{eq:est_in_out} yields \eqref{eq:est_u_eps_D_extreme}.

\end{proof}

\subsection{Analysis of the averaged adjoint equation}
We introduce for $\eps\in [0,\tau]$  the (not necessarily symmetric) bilinear form $b_\eps:H^1_0(\Dsf) \times H^1_0(\Dsf)\to \VR$ by 
\ben
b_\eps(\psi, \varphi) := \int_{\Dsf\setminus \overbar \omega_\eps} \beta_2 \nabla \psi\cdot \nabla \varphi + \left(\int_0^1 \varrho_2'(su_\eps+(1-s)u)\;ds\right) \varphi \psi \;dx.
\een
Then the averaged adjoint equation \eqref{eq:aa_equation} for the Lagrangian $G$ given by \eqref{eq:lagrange_scale_eps_extreme}  reads: for $(u_0,u_\eps)\in X(0)\times X(\eps)$ find $q_\eps\in H^1_0(\Dsf)$, such that
\ben\label{eq:perturbed_adjoint_topo_extreme}
b_\eps(\psi,\fp^\eps)= - \int_{\Dsf\setminus \overbar \omega_\eps} (u+u_\eps) \psi \;dx
\een
for all $\psi \in H^1_0(\Dsf)$. In view of Assumption~\ref{ass:varrho_monotone} we have $\varrho_2'\ge 0$ and $\beta_2\ge \beta_m I$ and thus $b_\eps$ satisfies,
\ben
b_\eps(\psi,\psi) \ge \beta_m \|\nabla \psi\|_{L_2(\Dsf\setminus \overbar \omega)^d}^2  
\een
for all $\psi\in H^1_0(\Dsf)$ and $\eps \in [0,\tau]$. 
As for the state equation, we use the notation $q:=q^0$. 
\begin{lemma}
  \begin{itemize}
    \item[(i)] For each $\eps\in[0,\tau]$ equation \eqref{eq:perturbed_adjoint_topo_extreme} admits a solution whose restriction to $\Dsf\setminus \overbar \omega_\eps$ is unique. 
  \item[(ii)] For every $z\in \Dsf\setminus \overbar \Omega$ we find
    a number $\delta >0$, such that $q\in H^3(B_\delta(z))\subset C^1(\overbar B_\delta(z))$ for $d\in  \{2,3\}$. 
  \end{itemize}
\end{lemma}

The previous lemma shows that $Y(\eps,u,u_\eps)=\{q\in H^1_0(\Dsf):\; q=q_\eps \text{ a.e. on } \Dsf\setminus \overbar \omega_\eps  \}$ and thus Hypothesis~(H2) is satisfied. In the same way as done in \eqref{eq:extension} we extend the restriction $\fp_\eps|_{\Dsf\setminus \overbar \omega_\eps}$ (which is unique) to a function $\fp_\eps \in H^1_0(\Dsf)$ by solving the following Dirichlet problem: find $\fp^+_\eps\in H^1(\omega_\eps)$, such that
\ben
\begin{split}
  -\Div(\beta_2^\top \nabla \fp_\eps^+) + \varrho_2'(u)\fp &= -2u \quad \text{ in } \omega_\eps\\
  \fp_\eps^+& = \fp_\eps \quad \text{ on } \partial \omega_\eps.
\end{split}
\een
 With this function we define again 
\ben
\fp_\eps := \left\{
  \begin{array}{cl}
    \fp_\eps^+ & \text{ in } \omega_\eps\\
    \fp_\eps   & \text{ in } \Dsf\setminus \overbar \omega_\eps
\end{array}\right..
\een
It is clear that $\fp_\eps\in Y(\eps,u,u_\eps)$.  We proceed with a H\"older-type estimate for the extension $\eps \mapsto q_\eps$.

\begin{lemma}\label{lem:dif_p_peps_extreme}
	There is a constant $C>0$, such that for all small $\eps >0$,
	\ben\label{eq:p_pesp}
	\|q_\eps-q\|_{H^1(\Dsf)}\le C(\|u_\eps-u\|_{L_2(\Dsf)} + \eps^{d/2}).
	\een
\end{lemma}
\begin{proof}
  The proof is the same as the one of Lemma~\ref{lem:dif_p_peps} and therefore omitted. 
 \end{proof}
It is readily checked that Hypothesis~(H3) is satisfied, too.
\begin{lemma}
We have 
  \ben\label{eq:hypo_H2_extreme}
	  \begin{split}
	    \lim_{\eps \searrow 0}\frac{G(\eps, u,q) - G(0,u,q)}{\ell(\eps)} = &  (-\beta_2 \nabla u\cdot \nabla q   - \varrho_2(u) q   + f_2\fp)(z).
	\end{split}
	\een
\end{lemma}
\begin{proof}
 Since $f_2\in C(\overbar B_\delta(z))$ and $u,\fp\in C^1(\overbar B_\delta(z))$ for a small $\delta>0$, we can repeat the steps of the proof of Lemma~\ref{lem:hypo_H2}.
\end{proof}

\subsection{Variation of the averaged adjoint equation and its weak limit}
The next step is to consider the variation of the averaged adjoint state. The variation of the 
averaged adjoint variable, denoted $Q^\eps$,  is defined as in Definition~\ref{def:variation_q}.

The following is the analogue of Lemma~\ref{lem:bounds_scale_one} with the main difference that we have an 
additional equation which gives information of $Q$ inside the inclusion $\omega$. 
	\begin{lemma}\label{lem:bounds_scale_one_extreme}
		There is a constant $C>0$, such that for all small $\eps >0$, 
		\ben\label{eq:bound_p_eps_extreme}
		\int_{\VR^d} (\eps \fP^\eps)^2+ |\nabla \fP^\eps|^2 \; dx \le C.
		\een
	\end{lemma}
	\begin{proof}
	  We follow the steps of Lemma~\ref{lem:bounds_scale_one}, but use Lemmas~\ref{lem:dif_p_peps_extreme},\ref{lem:u_ueps_extreme} instead Lemmas~\ref{lem:dif_p_peps},\ref{lem:u_ueps}.  
	\end{proof}
      
      \begin{theorem}\label{thm:convegence_peps_extreme}
	We have
	\begin{subequations}
	  \begin{align}
	    \label{eq:weakQa} \nabla \fP^\eps \rightharpoonup \nabla \fP \qquad & \text{ weakly in } L_2(\VR^d)^d,\\
	  \label{eq:weakQb} \eps \fP^\eps \rightharpoonup 0 \qquad & \text{ weakly in } H^1(\VR^d),
	\end{align}
      \end{subequations}
	  where  $[\fP]\in \dot{BL}(\VR^d)$ is the unique solution to 
	\begin{subequations}\label{eqeq}
	  \begin{align}
	    \label{eq:exterior_beta_extreme} \int_{\VR^d\setminus\overbar \omega} \beta_2(z) \nabla \psi\cdot \nabla \fP  \; dx & = \int_\omega \zeta \cdot \nabla \psi\; dx && \text{ for all } \psi \in BL(\VR^d),\\
	    \label{eq:exterior_beta_extremeb}\int_{\omega} \beta_2(z)\nabla \psi\cdot \nabla \fP\;dx & =0 && \text{ for all } \psi \in H^1_0(\omega), 
	  \end{align}
	\end{subequations}
where  $\zeta := -(\beta_1(z)-\beta_2(z))\nabla \fp(z)$.
			\end{theorem}
	\begin{proof} 
	  It follows from Lemma~\ref{lem:bounds_scale_one_extreme} that for every null-sequence $(\eps_n)$ there is a subsequence (indexed the same) and $\fP\in \text{BL}(\VR^d)$ such that \eqref{eq:weakQa} and \eqref{eq:weakQb} holds for this subsequence. Now using the same arguments as in the proof of Theorem~\ref{thm:convegence_peps} shows that $\fP$ satisfies \eqref{eq:exterior_beta_extreme}. The uniqueness of $Q|_{\VR^d\setminus \overbar \omega}$ follows directly from \eqref{eq:exterior_beta_extreme}.  To prove \eqref{eq:exterior_beta_extremeb} note that 
	  $Q^{\eps_n}$ satisfies 
	  \ben
	  \int_\omega \beta(T_{\eps_n}(x))\nabla \psi \cdot \nabla \fP^{\eps_{n}} \;dx =0\quad \text{ for all }    \psi\in H^1_0(\omega). 
	  \een
	  Using \eqref{eq:weakQa} and \eqref{eq:weakQb} we may pass to the limit $n\to \infty$ which shows that $Q$ satisfies \eqref{eq:exterior_beta_extremeb}. Since $Q|_{\partial \omega}$ is uniquely determined from \eqref{eq:exterior_beta_extreme} also \eqref{eq:exterior_beta_extremeb} admits a unique solution, because it is a Dirichlet problem with boundary values $\fP|_{\partial \omega}$. 
\end{proof}

We are now ready to compute  $R(u,\fp)$. 
\begin{lemma}\label{lem:verify_H3_extreme}
We have 
\ben\label{eq:form_R}
R(u,\fp) = -\beta_2(z) \nabla u(z)\cdot \fint_{\omega} \nabla \fP\; dx,
\een
where $[\fP]$ is the solution to \eqref{eq:exterior_beta_extremeb}.
\end{lemma}
\begin{proof}
 The proof follows the lines of Lemma~\ref{lem:verify_H3} and Lemma~\ref{lem:u_ueps_extreme}.
\end{proof}

Collecting all previous results we see that Theorem~\ref{thm:diff_lagrange} can be applied and we obtain the following result.

\begin{theorem}\label{thm:main_extreme}
  The topological derivative of $J$ given by \eqref{eq:cost_extreme} in $z\in \Dsf$ is given by 
  \ben\label{eq:topo_final_extreme}
  \lim_{\eps\searrow 0} \frac{J(\omega_\eps(z))-J(\Omega)}{|\omega_\eps|} =  \partial_\ell G(0, u,\fp) + R(u,\fp),
\een
where 
\ben
\partial_\ell G(0, u,\fp) = ( -\beta_2\nabla u\cdot \nabla \fp  - \varrho_2(u)\fp +  f_2\fp)(z)
\een
and 
\ben\label{eq:form_R_main_extreme}
R(u,\fp) = -\beta_2(z) \nabla u(z)\cdot\fint_{\omega} \nabla \fP\; dx
\een
and $\fP$ depends on $z$ and is the unique solution to \eqref{eq:exterior_beta_extreme}.  
\end{theorem}

Let $\fP_\zeta$ denote the solution to \eqref{eq:exterior_beta_extreme}-\eqref{eq:exterior_beta_extremeb} for fixed $z\in \Dsf$ and for $\zeta \in \VR^d$. Since $Q_\zeta$ depends linearly on $\zeta$ we can proceed as in Subsection~\ref{subsec:topo} and introduce a polarisation matrix $\Cp\in \VR^{d\times d}$ (depending on $\beta_2(z)$) such that $ \Cp\zeta =\fint_{\omega} \nabla \fP_\zeta\;dx$  to simplify \eqref{eq:topo_final_extreme}. Finally in the same way done as in Lemmas~\ref{lem:pol_sym},\ref{lem:pol_pos} we can show that $\Cp$ is symmetric if $\beta_2$ is symmetric and that it is always positive definite. Since the considerations are almost identical with the ones of Subsection~\ref{subsec:topo} the details are left to the reader.

\subsection*{Concluding remarks}
	In this paper we showed that the Lagrangian averaged adjoint framework of \cite{a_DEST_2017a} provides an efficient and fairly simple tool to compute topological derivatives for semilinear problems. We illustrated that using standard apriori estimates for the perturbed states and averaged adjoint variables are sufficient to obtain the topological sensitivity under comparatively mild assumptions on the inclusion. Our work 
	also provides a second examples (the first was given by \cite{c_DEST_2016a})  for which the $R$ term in \cite[Thm. 3.1]{a_DEST_2017a} is not equal to zero and thus underlines the flexibility of this theorem.

There are several problems that remain open for further research. Firstly, it would be interesting to consider quasilinear equations, but also other types of equations, such as Maxwell's equation. Secondly, an important point we have not addressed here is the topological derivative when Dirichlet boundary conditions are imposed on the inclusion. This case is know to be much different from the Neumann case and needs further investigations. 

	\section{Appendix}

	
	\subsection{The space \texorpdfstring{$\VE_p(\VR^d)$}{Lg}}
Define for $1<p<d$ the space
\ben
\VE_p(\VR^d) := \{u\in L_{p^*}(\VR^d):\; \nabla u \in L_2(\VR^d)^d\}
\een
with the norm
\ben
\|u\|_{\VE_p} := \|u\|_{L_{p^*}(\VR^d)} + \|\nabla u\|_{L_2(\VR^d)^d}.
\een

	\begin{lemma}
	  Let $d\ge 3$. Let $A$ satisfy \eqref{eq:laplace_unbounded_intro} and $A=A^\top$ a.e. on $\VR^d$. Then for every $F\in \mathcal L(\VE_2(\VR^d), \VR)$, there is a unique solution $\fP \in \VE_2(\VR^d)$ to
    \ben\label{eq:topo_semi_final4}
    \int_{\VR^d} A\nabla \varphi \cdot \nabla \fP \;dx  =    F(\varphi) \quad \text{ for all }\varphi \in \VE_{2}(\VR^d).
		\een
  \end{lemma}
  \begin{proof}
    Let us introduce the energy $\Ce:\VE_2(\VR^d) \to \VR$ by
\ben
\Ce(\varphi) := \frac12\int_{\VR^d} A \nabla \varphi\cdot \nabla\varphi \;dx  -  F(\varphi).
\een
We are now going to prove that the minimisation problem
\ben
\inf_{\varphi \in \VE_2(\VR^d)} \Ce(\varphi),
\een
admits a unique solution. We have to show that $\Ce$ is coercive on $\VE_2(\VR^d)$ , that is,
\ben
 \lim \Ce(\varphi) = +\infty \quad \text{ for  } \varphi \in \VE_2(\VR^d), \text{ with }  \|\varphi\|_{\VE_2} \to \infty 
\een
and that the energy is lower semi-continuous; see \cite[Prop. 1.2, p.35]{b_EKTE_1999a}. For the coercively it is sufficient to show that there are constants $C_1,C_2>0$ such that
\ben\label{eq:coercive_E}
\Ce(\varphi) \ge C_1\|\varphi\|_{\VE_2(\VR^d)}^2 - C_2 \|\varphi\|_{\VE_2(\VR^d)}\quad \text{ for all } \varphi \in \VE_2(\VR^d).
\een
Using the NSG inequality we can estimate as follows
\ben\label{eq:energy_a}
\begin{split}
  \frac12\int_{\VR^d} A \nabla \varphi \cdot \nabla \varphi \;dx &\ge \frac12 M_1 \|\nabla \varphi\|_{(L_2)^d}^2 \\
	       & \ge \frac14 C_N^2 M_1 \|\varphi\|_{L_{2^*}}^2 + \frac14 M_1 \|\nabla \varphi\|_{(L_2)^d}^2 \\
	       & \ge C(\|\varphi\|_{L_{2^*}}^2 + \|\nabla \varphi\|_{(L_2)^d}^2),
\end{split}
\een
where $C:=\min\{\frac14 C_N^2M_1, \frac14 M_1\}$. On the other hand using again the NSG inequality yields
\ben\label{eq:energy_b}
\begin{split}
  (\|\varphi\|_{L_{2^*}} + \|\nabla \varphi\|_{(L_2)^d})^2 & = \|\varphi\|_{L_{2^*}}^2 + \|\nabla \varphi\|_{(L_2)^d}^2 + 2\|\varphi\|_{L_{2^*}}\|\nabla \varphi\|_{(L_2)^d}\\
							   & \le \|\varphi\|_{L_{2^*}}^2 + \|\nabla \varphi\|_{(L_2)^d}^2 + 2\frac{1}{C_N}\|\varphi\|_{L_{2^*}}^2\\
							   & \le \tilde C (\|\varphi\|_{L_{2^*}}^2 + \|\nabla \varphi\|_{(L_2)^d}^2)
\end{split}
\een
where $\tilde C:= \min\{ 1+ \frac{2}{C_N}, 1\}$. Combining \eqref{eq:energy_a} and \eqref{eq:energy_b} yields
\ben\label{eq:energy_c}
\frac12\int_{\VR^d}  A\nabla \varphi\cdot \nabla \varphi \;dx  \ge \frac{C}{\tilde C} \|\varphi\|^2_{\VE_2(\VR^d)}.
\een
Finally the continuity of $F$ gives
\ben\label{eq:energy_d}
\begin{split}
  F(\varphi)  & \ge - \|F\|_{\mathcal L(\VE_2,\VR)}\|\varphi\|_{\VE_2(\VR^d)}. 
\end{split}
\een
Combining \eqref{eq:energy_c} and \eqref{eq:energy_d} yields \eqref{eq:coercive_E} with $C_1 = C/\tilde C$ and $C_2 = \|F\|_{\mathcal L(\VE_p,\VR)}$.  
  \end{proof}

Recall the Gagliardo-Nirenberg-Sobolev inequality (short NSG inequality)
	\ben\label{eq:nirenberg}
	\|u\|_{L_{p^*}(\VR^d)} \le C_N\|\nabla u\|_{L_p(\VR^d)}
	\een
	valid for all $u \in C^\infty_c(\VR^d)$. The constant $C_N$ does not depend on the 
	support of the function $u$. Notice also that for $p=d$ the inequality fails. Thanks to Lemma~\ref{lem:space_E} we know that $C^\infty_c(\VR^d)$ is dense in $\VE_p(\VR^d)$. Hence it follows that \eqref{eq:nirenberg} holds for all test functions $u \in \VE_p(\VR^d)$.  For instance for $d=3$ and $E_2(\VR^3)$ we have
	\ben
	\|u\|_{L_6(\VR^3)} \le C\|\nabla u\|_{L_{2}(\VR^3)}.
	\een

\begin{lemma}\label{lem:space_E}
	For all $1<p<d$ the space $(\VE_p(\VR^d),\|\cdot\|_{\VE_p})$ is a Banach space. For every sequence $(u_n)$ in $\VE_p(\VR^d)$ we find a subsequence $(u_{n_k})$ and an element $u\in \VE_p(\VR^d)$, such that 
             \ben
		\begin{split}
		  u_{n_k} & \rightharpoonup u  \quad \text{ weakly in } L_{p^*}(\VR^d) \quad \text{ as }  n\to \infty,\\
		  \nabla u_{n_k} & \rightharpoonup \nabla u  \quad \text{ weakly in } L_p(\VR^d)^d \quad \text{ as } n\to \infty. 
		\end{split}
		\een
	Moreover, $C^\infty_c(\VR^d)$ is dense in $\VE_p(\VR^d)$.  
      \end{lemma}
 \begin{proof}
	Let $(u_n)$ be a bounded sequence in $\VE_p(\VR^d)$. Since the $L_p(\VR^d)$-spaces are reflexive for all $p\in (1,\infty)$, we 
	find elements $\eta \in L_{p^*}(\VR^d)$ and $\zeta \in L_p(\VR^d)^d$ and a subsequence $(u_{n_k})$,  such that
		\ben
		\begin{split}
		  u_{n_k} & \rightharpoonup \eta  \quad \text{ weakly in } L_{p^*}(\VR^d) \quad \text{ as }  n\to \infty,\\
		  \nabla u_{n_k} & \rightharpoonup \zeta  \quad \text{ weakly in } L_p(\VR^d)^d \quad \text{ as } n\to \infty. 
		\end{split}
		\een
Now we claim that $\zeta = \nabla \eta$, which then implies $\eta\in \VE_p(\VR^d)$. To see this notice that by definition of the weak derivative
		\ben\label{eq:weak_derivative_definition_p_eps}
		\int_{\VR^d} \partial_{x_i}\varphi u_{n_k} \; dx = - \int_{\VR^d} \varphi \partial_{x_i} u_{n_k} \; dx
		\een
		for all $\varphi \in C^\infty_c(\VR^d)$. Now we pass to the limit in \eqref{eq:weak_derivative_definition_p_eps} and obtain
		\ben
		\int_{\VR^d} \partial_{x_i}\varphi \eta \; dx = - \int_{\VR^d} \varphi \zeta \; dx
		\een
		for all $\varphi \in C^\infty_c(\VR^d)$, which proves the claim. Since a linear and continuous functional on a Banach space is continuous if and only if it is weakly continuous the claim follows.

To prove the completeness of $\VE_p(\VR^d)$ let $(u_n)$ be a Cauchy sequence in $\VE_p(\VR^d)$. Then $(u_n)$ is a Cauchy sequence in $L_{p^*}(\VR^d)$ and 
$(\nabla u_n)$ is a Cauchy sequence in $L_p(\VR^d)^d$. Since $(u_n)$ is a Cauchy sequence in $L_{p^*}(\VR^d)$ and $(\nabla u_n)$ is a Cauchy sequence in $L_p(\VR^d)$ we find elements $\eta \in L_{p^*}(\VR^d)$ and $\zeta\in L_p(\VR^d)^d$ so that $u_n \to \zeta$ strongly in $L_{2^*}(\VR^d)$ and $\nabla u_n \to \zeta $ strongly in $L_p(\VR^d)^d$. Now we can follow the previous argumentation to show that $\nabla \eta = \zeta$ which shows that $\eta\in \VE_p(\VR^d)$ and thus shows that $\VE_p(\VR^d)$ is complete. 	

Let us now show the density of $C^\infty_c(\VR^d)$ in $\VE_p(\VR^d)$. As shown in \cite[Thm. 3.22, p. 68]{b_ADFO_2003a} it suffices to show every $u\in \VE_p(\VR^d)$ with bounded support can 
be approximated by function in $C^\infty_c(\VR^d)$. Suppose that the support of 
$u$ is compact. Denote by $u_\eps := \varrho_\eps \ast u$ the standard mollification of $u$ with a mollifier $\varrho_\eps\in C^\infty(\VR^d)$, $\eps >0$; see \cite[pp. 36]{b_ADFO_2003a}. Then $u_\eps \in L_{p^*}(\VR^d)$ and
$\partial_i u_\eps = \varrho_\eps\ast \partial_i u\in L_2(\VR^d)$. Moreover according to  \cite[Thm. 2.29, p. 36]{b_ADFO_2003a} we have $\lim_{\eps \searrow 0} \|u_\eps - u\|_{L_{2^*}(\VR^d)}=0$ and $\lim_{\eps \searrow 0} \|\partial_{x_i}(u_\eps - u)\|_{L_{2}(\VR^d)}=0$. This finishes the proof. 
      \end{proof}

\bibliography{mybib}
\bibliographystyle{plain}
\end{document}

%% file: slidesymbols.tex
\providecommand{\Div}{\operatorname{div}}          

\providecommand{\Id}{\Op{Id}}                     

\newcommand{\VE}{{\mathbf{E}}}

\newcommand{\VR}{{\mathbf{R}}}

\newcommand{\VV}{{\mathbf{V}}}

\newcommand{\Dsf}{\mathsf{D}}

\providecommand{\Ce}{{\cal E}}

\providecommand{\Cp}{{\cal P}}

\providecommand{\Cr}{{\cal R}}

\newcommand*{\Op}[1]{\mathsf{#1}} 

